\date{January 25, 2013}

\documentclass[12pt]{amsart}
\usepackage{latexsym,amsmath,amsfonts,amscd,amssymb}
\usepackage{graphics}
\newtheorem{theorem}{Theorem}[section]
\newtheorem{lemma}[theorem]{Lemma}
\newtheorem{proposition}[theorem]{Proposition}
\newtheorem{corollary}[theorem]{Corollary}

\newtheorem{definition}[theorem]{Definition}

\theoremstyle{remark}
\newtheorem{remark}[theorem]{Remark}

\newcommand{\la}{\langle}
\newcommand{\ra}{\rangle}

\newcommand{\Arg}{\text{Arg}}

\newcommand{\unit}{\textbf{1}}

\newcommand{\Div}{\operatorname{Div}}

\newcommand{\cD}{{\mathcal D}}

\newcommand{\cL}{{\mathcal L}}
\newcommand{\cO}{{\mathcal O}}
\newcommand{\cP}{{\mathcal P}}

\newcommand{\cS}{{\mathcal S}}

\newcommand{\CC}{{\mathbb C}}

\newcommand{\HH}{{\mathbb H}}
\newcommand{\NN}{{\mathbb N}}

\newcommand{\QQ}{{\mathbb Q}}
\newcommand{\RR}{{\mathbb R}}

\newcommand{\ZZ}{{\mathbb Z}}

\newcommand{\bk}{{\mathbf{k}}}

\title{Poisson-Newton formulas and Dirichlet series}


\author[V. Mu\~{n}oz]{Vicente Mu\~{n}oz}
\address{Facultad de
Matem\'aticas, Universidad Complutense de Madrid, Plaza de Ciencias
3, 28040 Madrid, Spain}
\email{vicente.munoz@mat.ucm.es}

\author[R. P\'{e}rez Marco]{Ricardo P\'{e}rez Marco}
\address{CNRS, LAGA UMR 7539, Universit\'e Paris XIII,
99, Avenue J.-B. Cl\'ement, 93430-Villetaneuse, France}
\email{ricardo@math.univ-paris13.fr}

\thanks{Partially supported through Spanish MICINN grant MTM2010-17389.}

\begin{document}

\maketitle

\begin{abstract}
  We prove that a Poisson-Newton formula, in a broad sense, is associated to each Dirichlet
series with a meromorphic extension to the whole
complex plane. These formulas simultaneously generalize the classical Poisson formula and Newton
formulas for Newton sums.
Classical Poisson formulas in Fourier analysis, classical summation formulas as
Euler-McLaurin or Abel-Plana formulas, explicit formulas in number
theory and Selberg trace formulas in Riemannian
geometry appear as special cases of our general Poisson-Newton formula. We also associate to finite order
meromorphic functions general Poisson-Newton formulas that yield many classical 
integral formulas.
\end{abstract}

\noindent \emph{We dedicate this article to Daniel Barsky and Pierre Cartier
 for their interest and constant support}

\section{Introduction}

All classical Poisson formulas for functions in Fourier analysis result from the general distributional Poisson formula
\begin{equation}\label{eqn:a1}
\sum_{n\in \ZZ} e^{i\frac{2\pi}{\lambda} n t}  =\lambda \sum_{k\in \ZZ}  \delta_{\lambda k} \, ,
\end{equation}
which is an identity of distributions identifying an infinite sum of
exponentials, converging in the sense
of distributions, and a purely atomic distribution. This distributional formula
is related to the simplest finite
Dirichlet series
$$
f(s)=1-e^{-\lambda s} \ .
$$
It is interesting to observe that on the left hand side of (\ref{eqn:a1})
we have an exponential sum 
$$
W(f)=\sum_\rho e^{\rho t} \ ,
$$
where the sum runs over the zeros $\rho_n = \frac{2\pi i}{\lambda} n$, $n\in \ZZ$ of $f$,
and on the right hand side of (\ref{eqn:a1})
we have a sum of atomic masses at the multiples of the fundamental frequency $\lambda$.
One can say that the frequencies associated to the zeros are resonant at the fundamental frequencies. Taking
the Fourier transform we obtain the dual Poisson formula that is of the same form where we exchange zeros
and fundamental frequencies. Thus the fundamental frequencies are also resonant at the zeros.

The main purpose of this article is to show that this is general and to each meromorphic Dirichlet series $f$
we can associate a
distributional Poisson formula
\begin{equation}\label{eqn:a2}
W(f)=\sum_\rho n_\rho e^{\rho t} =\sum_{\bk} \langle \boldsymbol{\lambda} , \bk\rangle  b_\bk \,\delta_{\langle \boldsymbol{\lambda} , \bk\rangle} \, ,
\end{equation}
where the first sum of exponentials runs over the divisor of $f$, i.e., zeros and poles $\rho$ with multiplicities $n_\rho$, and the second sum runs over non-zero sequences $\bk=(k_1, k_2,\ldots )\in
\NN^\infty$ of non-negative integers, all of them zero but finitely many,
and $\langle \boldsymbol {\lambda} , \bk\rangle = \sum \lambda_j k_j$. The
equality holds in $\RR_+^*$. Conversely, we prove that
any such Poisson formula comes from
a Dirichlet series.

\medskip

The distribution
$$
W(f)=\sum_\rho n_\rho e^{\rho t}
$$
is well defined in $\RR_+^*$ and is called the \emph{Newton-Cramer distribution} of $f$.
We name it after Newton because it appears as a distributional
interpolation of the Newton sums to exponents $t\in \RR$, since in the complex
variable\footnote{The variable $z=e^s$ or better $z=e^{-s}$ is the proper variable when
dealing with Dirichlet series.}
$z=e^{s}$ the zeros are the $\alpha =e^{\rho}$ so
$$
W(f)(t)=\sum_\alpha \alpha^{ t} \, ,
$$
and for integer values $t=m\in \ZZ$ we get (in case of convergence) the Newton sums
 $$
W(f)(m)=S_m=\sum_\alpha \alpha^{m} \, .
$$
There is a precise theorem behind this observation. We show that
our Poisson-Newton formula for a finite Dirichlet series  $f$ with a single fundamental 
frequency is strictly equivalent to
the classical Newton relations. This is the reason why we name also after Newton our general Poisson formulas.

\medskip

Writing $\rho=i\gamma$ we see that the sum $W(f)$ of the left hand side of (\ref{eqn:a2})
is the Fourier transform of the atomic
Dirac distributions $\delta_\gamma$ and we can formally write
$$
\sum_\gamma n_\rho \hat \delta_\gamma =\sum_{\bk} \langle \boldsymbol {\lambda} ,
\bk\rangle  b_\bk \,\delta_{\langle \boldsymbol {\lambda} , \bk\rangle} \, .
$$
The form of this formula, relating zeros to fundamental frequencies,
strongly reminds other distributional formulas in other contexts.
In number theory, more precisely in the theory of zeta and $L$-functions,
the same type of identities do appear as ``explicit formulas''
associated to non-trivial zeros of the zeta and other $L$-functions. These \textit{explicit formulas}, when
written in distributional form, reduce to a single distributional relation that identifies a sum of
exponentials associated to the divisor of the zeta or $L$-function and an atomic distribution associated
to the location of prime numbers. Usually the sum runs over non-trivial zeros, and the sum over trivial zeros appears
hidden in other forms as a \textit{Weil functional}, which is classically interpreted as 
corresponding to the ``infinite prime''\footnote{It may be more appropriate to talk of the prime $p=1$.}.
For that reason, Delsarte labeled this formula as ``Poisson formula
with rest'' (see \cite{D}), the ``rest'' refers to the sum over the trivial part of the divisor. More precisely, for
the Riemann zeta function, we have in $\RR_+^*$
$$
\sum_\rho n_\rho e^{\rho t} +W_0(f) =\sum_p \sum_{k\geq 1}  \log p \, \delta_{k\log p} \ ,
$$
where the sum on the left runs over the non-trivial (i.e., non-real) zeros $\rho$, and the sum over $p$ runs over prime numbers. Conjecturally, the non-trivial zeros are simple, i.e., $n_\rho=1$. The term $W_0(f)$ is the sum over the
trivial (real) divisor and is computable
$$
W_0(f)(t)=-e^t+\sum_{n\geq 1} e^{-2n t} =-e^t + \frac{1}{e^{2t} -1} \ ,
$$
and corresponds to Delsarte ``rest'', or to the Weil functional of the infinite
prime. Also we have in this case
$$
\sum_\rho n_\rho e^{\rho t}=e^{t/2} V(t) + e^{t/2} V(-t) \ ,
$$
where
$$
V(t)=\sum_{\Re \gamma >0} e^{i\gamma t}  \ ,
$$
is the classical Cramer function, studied by H. Cramer \cite{C}, where $\rho=\frac12+i\gamma$. 
This motivates that
we name our distribution $W(f)$
also after Cramer.

\medskip

In Riemannian geometry, we have the same structure
for the Selberg trace formula for compact surfaces with constant negative curvature. With the relevant
difference that Selberg zeta function is of order $2$, which gives a ``rest'' of order $2$ also.
Selberg formula relates the length
of primitive geodesics, which play the role of prime
numbers, and the eigenvalues of the Laplacian, which
give the zeros of the Selberg zeta function. For non-negative constant curvature, the formulas are of
a different nature and the distribution on the right side are no longer simple atomic Diracs, but also higher order derivatives appear.
This will be discussed elsewhere. In the context of dynamical systems and semiclassical quantization,
we have Gutzwiller
trace formula, which relates the structure of the periodic orbits
of a classical mechanical system to the energy levels of the associated quantum system. 

\medskip

The interpretation and analogy of these formulas with ``Poisson formulas'' was noticed
long time ago. We should
mention in particular the classical work of A.P. Guinand \cite{G}, 
J. Delsarte \cite{D}, A. Weil \cite{W}, and
results related to Hamburger theorem \cite{H}, \cite{KM}.
Already the title of Delsarte's
article points to the Poisson flavor of these formulas
\textit{``Formules de Poisson avec reste''}. More recently
this analogy between Poisson and explicit formulas and its relation with zeta-regularization
is studied for the Selberg trace formula for surfaces with constant negative curvature by
P. Cartier and A. Voros in \cite{CV}. General Poisson formulas for Riemannian manifolds relating the spectrum of positive elliptic operators and the length spectrum were developped by J. Chazarain \cite{CH} and J.J. Duistermaat and V.W. Guillemin \cite{DG}, and the Dirichlet series associated to the spectrum of the heat equation by S. Minakshisundaram and
\AA. Pleijel \cite{MP}, after foundational work by T. Carleman \cite{CA}.

\medskip

Our goal is to put in the proper context, generalize and
make precise the analogy of Poisson and trace formulas, and derive a general class
of Poisson formulas that contain all such instances. More precisely, to each meromorphic Dirichlet series of
finite order we associate a Poisson-Newton formula. All relevant known formulas can be generated in this way.
On the other hand the fact that \textit{explicit formulas} in number theory and Selberg trace formula can be
seen as a generalization of Newton formulas, seems to be a new interpretation.

\medskip

It is important to remark that in our general setting the Poisson-Newton formulas are
independent from a possible functional equation for the
Dirichlet series $f$, contrary to what happens in classical formulas.
As a matter of fact, we do
associate a Poisson-Newton formula to any Dirichlet series with no functional equation.
This is sometimes hidden in the classical theory where explicit formulas and functional equations
come hand to hand.
For instance, most proofs in the literature derive
the explicit formula for the zeta function using its functional
equation, and the ``rest'' term, borrowing Delsarte
terminology, is computed from the fudge factor from the functional equation.
Nevertheless our approach
shows that the functional equation is not related to the existence of an explicit formula. 
Moreover, for an $f$ having
a functional equation, the
``rest'' term in our Poisson-Newton formula emerges from the non-symmetric part of the divisor of $f$.

Although independent, these questions are interrelated. As is well known, it is a classical
and basic  procedure since Riemann foundational memoir \cite{R}
to derive functional equations for the Dirichlet series
of $f$ from Poisson formulas for other Dirichlet series (for the $\theta$-function in the case of the Riemann
zeta function).

\medskip

We also derive a general Poisson-Newton formula associated to finite order meromorphic functions $f(s)$ which
are not necessarily Dirichlet series, but have their divisor contained in a left half plane. In this general
situation, the Newton-Cramer distribution is still defined by the exponential series
(converging in $\RR_+$ as distribution)
$$
W(f)(t)=\sum_\rho e^{\rho t} \, .
$$
This time, the distribution is no longer a sum of purely atomic measures in $\RR_+^*$.
A particularly important case is when
the divisor of $f$ if left-oriented, i.e., contained in a left cone. Then the Newton-Cramer distribution is a
$\theta$-distribution and it is a distribution given by an analytic function in $\RR_+^*$. An application of our
general Poisson-Newton formula gives a collection of classical formulas: Gauss formula for the logarithmic
derivative of the $\Gamma$-function, Binet formula for the logarithm of the $\Gamma$-function, general
Gauss and Binet formulas for higher order Barnes $\Gamma$-functions, etc.

\medskip

The structure at $0$ of the distributions appearing in the Poisson-Newton formula is interesting. In the construction
of Newton-Cramer distribution we have some parameter freedom that is irrelevant for the structure of the 
distribution in $\RR^*$, but not at $0$. But precisely the variation of this parameter gives a distributional form 
of different infinite Euler-MacLaurin type formulas associated to each Dirichlet series. 
More precisely, the classical 
Euler-MacLaurin formula (as well as Abel-Plana summation formula) can be derived from the simplest case of 
the Dirichlet series $f(s)=1-e^{-s}$. The distributional infinite Euler-MacLaurin formula 
sheds some light on Ramanujan's theory of the ``constant'' of a diverging series. Most of the other summation 
formulas over the semi-group generated by the frequencies of the Dirichlet series seem new.

\section{Dirichlet series}\label{sec:2}

We consider a non-constant Dirichlet series
 \begin{equation}\label{eqn:1}
 f(s)=1+\sum_{n\geq 1} a_n \ e^{-\lambda_n s} \ ,
 \end{equation}
with $a_n \in \CC$ and
$$
0< \lambda_1 < \lambda_2 < \ldots
$$
with $(\lambda_n)$ a finite set (equivalently, take
the sequence $(a_n)$ with all but finitely many elements being zero)
or $\lambda_n \to +\infty$, such that we have a half
plane of absolute convergence (see \cite{HR} for background on Dirichlet series), i.e., for some $\bar \sigma \in \RR$ we have
$$
\sum_{n\geq 1} |a_n | \ e^{-\lambda_n \bar \sigma} <+\infty \, .
$$
It is classical (\cite{HR}, p.8) that
$$
\bar \sigma=\limsup \frac{\log (|a_1|+|a_2|+\ldots +|a_n|)}{\lambda_n} \ .
$$

The Dirichlet series (\ref{eqn:1}) is therefore absolutely and uniformly
convergent on right half-planes $\Re s \geq \sigma$, for any $\sigma >\bar \sigma$.

We assume that $f$ has a meromorphic extension of finite order
to all the complex plane $s\in \CC$. We denote by $(\rho)$ the set of zeros and poles
of $f$, and the integer $n_\rho$ is the multiplicity of $\rho$ (positive for zeros and
negative for poles, with the convention $n_\rho =0$ if $\rho$ is neither a zero nor pole). 
The convergence exponent of $f$ is the minimum
integer $d\geq 0$ such that
$$
\sum_{\rho \not= 0} |n_\rho| \, |\rho|^{-d} < +\infty\ .
$$
We have $d=0$ if and only if $f$ is a rational function, which cannot be a Dirichlet series, thus
$d\geq 1$. Indeed we always have $d\geq 2$ (see Corollary \ref{cor:exp}).
The order $o$ of $f$ satisfies $ d \leq [o]+1$. 

Since $f$ has finite order, we have the Hadamard factorization of $f$ (see \cite{A}, p.208)
$$
f(s)=s^{n_0} e^{Q_f(s)} \prod_{\rho \not=0 } E_m (s/\rho )^{n_\rho} \ ,
$$
where $m=d-1\geq 0$ is minimal for the convergence of the product with
$$
E_m(z)=(1-z) e^{z+\frac12 z^2 +\ldots +\frac1m z^m} \ ,
$$
and $Q_f$ is a polynomial uniquely defined up to the addition of an integer multiple of $2\pi i$. The genus 
of $f$ is defined as the integer
$$
g=\min (\deg Q_f , m) \ ,
$$
and in general we have $d\leq g+1$ and 
$g\leq o \leq g+1$ (see \cite{A}, p.209). For a meromorphic Dirichlet series we prove that in fact $d=g+1$ 
(see Corollary \ref{cor:genus}). 

The origin plays no particular role, thus we may prefer 
to use Hadamard product with origin at some $\sigma \in \CC$,
$$
f(s)=(s-\sigma)^{n_\sigma} e^{Q_{f, \sigma} (s)} \prod_{\rho \not=\sigma } E_m \left (\frac{s-\sigma}{\rho -\sigma} \right )^{n_\rho} \ .
$$

We have, uniformly on $\Re s$,
$$
\lim_{\Re s \to +\infty} f(s) =1 \ ,
$$
thus 
$$
\sigma_1 =\sup_\rho \Re \rho <+\infty \ ,
$$
so $f(s)$ has neither zeros nor poles on the half plane $\Re s >  \sigma_1$. 
Sometimes in the applications $\sigma_1$ is a pole of $f$ because when the coefficients $(a_n)$ are real and positive 
then $f$ contains a singularity at $\bar \sigma$ by a classical theorem of Landau (see \cite{HR}, Theorem 10, p.10). 
The singularity is necessarily a pole by our assumptions, and in general $\sigma_1 =\bar \sigma$.

\medskip

Associated to the divisor $\mathrm{div}(f)= \sum n_\rho\,\rho$, we define a
distribution $W(f)= \sum n_\rho \, e^{\rho t}$ on $\RR^*_+$. We do this as follows.

Consider the space $\cS$ of $C^\infty$-functions of rapid decay on $\RR$ (i.e., $\varphi\in \cS$ if and only if
for any $n,m>0$, $|t|^n D^m \varphi \to 0$, as $t\to \pm \infty$). The dual space $\cS'$ is the
Schwartz space of tempered distributions.
As $\cD=C^\infty_0 \subset \cS$, we have that $\cS'\subset \cD'$, where $\cD'$ is the space of distributions.

\begin{lemma} \label{lem:1}
 For finite sets $A$, consider the family of locally integrable functions
 $$
 \tilde{W}_A(f)=\left( \sum_{\rho \in A} e^{\rho t}\right) {\emph\unit}_{\RR_+} \, .
 $$
 There is a family of distributions $W_A(f)$ which coincides with $\tilde{W}_A(f)$ in $\RR^*$, and which
 converges in $\RR$ (over the filter of finite sets $A$), to a distribution $W(f)$ in $\cD'$.

This distribution has support contained in $\RR_+$, is Laplace transformable, and
$e^{-\sigma_1 t} W(f)\in \cS'$. More precisely, if $\sigma_1$ is not a zero nor pole 
(resp. it is a zero or pole), 
$e^{-\sigma_1 t} W(f)$ is the (distributional)
$d$-th derivative of a uniformly bounded continuous function (resp. continuous function) 
on $\RR$ with support in $\RR_+$.

More precisely, we have 
$$
W(f) = e^{\sigma_1 t} \frac{D^d}{Dt^d} \left ((K_d(t)-K_d(0)) \emph \unit_{\RR_+}\right )  \ ,
$$
where 
$$
K_d(t)=\left ( n_{\sigma_1} \frac{t^d}{d!}\right ) \,  \emph \unit_{\RR_+} +\sum_{\rho \not= \sigma_1} \left( \frac{n_\rho}{(\rho-\sigma_1)^d} e^{(\rho-\sigma_1)t} \right) \emph \unit_{\RR_+} .
$$
\end{lemma}

\begin{proof}
We prove first the lemma when $\sigma_1$ is not a zero nor pole.
We define
 \begin{equation} \label{eqn:Kl}
K_\ell(t)= 
\sum_{\rho } \left( \frac{n_\rho}{(\rho-\sigma_1)^\ell} e^{(\rho-\sigma_1)t} \right) \unit_{\RR_+} .
   \end{equation}
Then for $\ell \geq d$, $K_\ell$ is absolutely convergent for $t\in \RR_+$ since
$$
\left |  e^{(\rho-\sigma_1)t} \right | =e^{\Re (\rho-\sigma_1) t} \leq 1 \ ,
$$
and $K_\ell$ is a uniformly bounded function in $\RR$, continuous in $\RR^*$, since
$$
|K_\ell| \leq \sum_\rho \frac{|n_\rho|}{|\rho-\sigma_1|^\ell} <\infty \ . 
$$
The function
 $$
  F_{\ell}(t)=(K_{\ell}(t)-K_{\ell}(0) ) \unit_{\RR_+}
 $$
is a uniformly bounded continuous function on $\RR$, for $\ell \geq d$.

For a finite set $A$, denote by
 $$
 K_{\ell,A}(t)= \sum_{\rho\in A} \left(
 \frac{n_\rho}{(\rho-\sigma_1)^\ell} e^{(\rho-\sigma_1)t}\right) \unit_{\RR_+}
 $$
the corresponding sum over $\rho\in A$, and $F_{\ell ,A}(t)=(K_{\ell,A}(t)-K_{\ell,A}(0)) \unit_{\RR_+}$.
On $\RR^*$,
 $$
 \tilde{W}_A(f)=
  \left(\sum_{\rho\in A} n_\rho \, e^{\rho t} \right)
  \unit_{\RR_+} =e^{\sigma_1 t}\frac{d^{d}}{dt^d} F_{d,A}(t) .
$$
We consider
 $$
 {W}_A(f)= e^{\sigma_1 t}\frac{D^{d}}{Dt^d} F_{d,A} ,
$$
taking the distributional derivative. 

For a smooth function (resp. function with polynomial growth) $K$ on $\RR$ and 
a test function $\varphi$ with compact support (resp. in the Schwarz class)
we have
\begin{align*}
\left \langle \frac{D}{Dt} (K \unit_{\RR_+} ), \varphi \right \rangle &= -\langle K \unit_{\RR_+} , \varphi'\rangle \\
&=-\int_0^{+\infty}  K(t) \varphi'(t) \, dt \\
&=-[K(t)\varphi(t)]_0^{+\infty} + \int_0^{+\infty}  K'(t) \varphi(t) \, dt \\
&=K(0) \varphi(0)+\langle K' \unit_{\RR_+} , \varphi \rangle \ ,
\end{align*}
thus
$$
\frac{D}{Dt} (K \unit_{\RR_+} ) =K' \unit_{\RR_+}  + K(0) \delta_0 \ .
$$

Then since $K'_{\ell , A}=K_{\ell -1, A}$ we get
\begin{align*}
 &\frac{D^{d}}{Dt^d} F_{d,A} =K_{0,A} (t)  + K_{1,A} (0) \delta_0 + K_{2,A} (0) \delta'_0 +\ldots +K_{d-1,A} (0) \delta_0^{(d-2)} \\ 
 &=K_{0,A} (t)  + \left (\sum_{\rho \in A} \frac{n_\rho}{\rho-\sigma_1}\right ) \delta_0 + 
\left (\sum_{\rho \in A} \frac{n_\rho}{(\rho-\sigma_1)^2}\right ) \delta'_0 + \ldots 
+ \left (\sum_{\rho \in A} \frac{n_\rho}{(\rho-\sigma_1)^{d-1}}\right ) \delta_0^{(d-2)} \ .
\end{align*}
Thus the difference between $\tilde{W}_A(f)$ and $W_A(f)$ is a distribution supported
at $\{0\}$.

We have the convergence $F_{d,A}\to F_{d}$, uniformly as continuous functions on $\RR$.
Thus we have the same limit $F_{d,A}\to F_{d}$ in the distributional sense.
Then taking the limit as distributions, $W_A(f)\to W(f)$, where
 $$
W(f)= e^{\sigma_1 t}\frac{D^{d}}{Dt^d} F_d\ ,
 $$
which is the $d$-th derivative of a uniformly bounded continuous function on $\RR$
with support on $\RR_+$, as stated.

When $\sigma_1$ is part of the divisor, then we do the same proof with
\begin{equation*} 
K_\ell(t)= \left ( n_{\sigma_1} \frac{t^\ell}{\ell !}\right ) \,   \unit_{\RR_+} +
\sum_{\rho \not= \sigma_1} \left( \frac{n_\rho}{(\rho-\sigma_1)^\ell} e^{(\rho-\sigma_1)t} \right) \unit_{\RR_+} ,
\end{equation*}
which adds to $W(f)$ a term $n_{\sigma_1} e^{\sigma_1 t}$.
\end{proof}

Note that we can write
$$
 W(f)|_{\RR_+^*} = \lim_A \tilde{W}_A(f)|_{\RR_+^*} = \sum_{\rho} n_\rho\,e^{\rho t} \, ,
 $$
as a distribution on $\RR_+^*$. But if $d\geq 2$, the family of distributions $\tilde {W}_A(f)$ is not converging to
a distribution in $\RR$ because the sums
$$
\sum_\rho  \frac{n_\rho}{(\rho-\sigma_1)^\ell} \ ,
$$
are not absolutely convergent for $\ell=1, \ldots , d-1$ (by the definition of $d$). On the
other hand, the same argument shows that for $\ell \geq d$, $K_{\ell , A}$ has a limit $ K_\ell$ in
the sense of distributions, and for any $k\geq 0$
$$
\frac{D^k}{Dt^k} F_{\ell , A} \to \frac{D^k}{Dt^k} F_{\ell }   \ .
$$

It is important to note that $W(f)|_{\RR_+^*}$ is independent of
the choices of $\sigma_1$ and of taking $d$ larger than the exponent of convergence.
The only dependence on $\sigma_1$ and $d$ is located at the structure of the distribution
at $0$.

\begin{proposition}
Let $d$ be the exponent of convergence of $f$ as before, and let $d'\geq d$.
We define for $\sigma \in \CC-\{\rho\}$, $\ell\geq 0$, and a finite subset $A\subset \{\rho \}$
 \begin{align*}
K_{\ell , A} (t, \sigma) &=\sum_{\rho \in A} \left( \frac{n_\rho}{(\rho-\sigma)^\ell} e^{(\rho-\sigma)t}
\right) \emph \unit_{\RR_+}  \ , \\
F_{\ell,A} (t,\sigma) &= (K_{\ell , A} (t, \sigma)- K_{\ell , A} (0, \sigma)) \emph \unit_{\RR_+}\ ,
 \end{align*}
then the following limits exist in the sense of distributions
 \begin{align*}
 \tilde W(f, d', \sigma) &= \lim_A e^{\sigma t} \frac{D^{d'}}{Dt^{d'}} K_{d' , A}(t,\sigma), \\
 W(f, d', \sigma) &= \lim_A e^{\sigma t} \frac{D^{d'}}{Dt^{d'}} F_{d' , A}(t,\sigma),
 \end{align*}
and $W(f, d', \sigma)|_{\RR^*}=\tilde W(f, d', \sigma)|_{\RR^*} = W(f)|_{\RR^*}$ is independent of $\sigma$ and $d'\geq d$.
\end{proposition}

\begin{proof}
We shall deal with the first case, the second one is similar.
The existence of the limit is proved as before. Take
 $$
 \tilde W_A(f,d',\sigma)= e^{\sigma t} \frac{D^{d'}}{Dt^{d'}} K_{d' , A} (t,\sigma) =
 e^{\sigma t} \frac{D^{d'}}{Dt^{d'}} \left( e^{(\sigma_1-\sigma)t} K_{d' , A} (t,\sigma_1) \right),
 $$
where $\sigma_1$ is the one considered before. Now, since $d'\geq d$, $K_{d' , A} (t,\sigma_1)$ converges to a distribution  $K_{d'}(t,\sigma_1)$. So
 $\tilde W_A(f,d',\sigma)$ converges as distribution to
$$
\tilde W(f, d', \sigma) = e^{\sigma t} \frac{D^{d'}}{Dt^{d'}} \left( e^{(\sigma_1-\sigma)t} K_{d'} (t,\sigma_1) \right)
= e^{\sigma t} \frac{D^{d'}}{Dt^{d'}} K_{d'} (t,\sigma).
$$

In $\RR_-^*$ the independence on $\sigma$ is clear since the distributions vanish.
In $\RR_+^*$, first we note that for $\ell \geq 0$, 
\begin{align*}
 \frac{\partial}{\partial t} K_{\ell , A} &=K_{\ell-1, A}
\end{align*}
and 
$$
\frac{D}{Dt} (K \unit_{\RR_+} ) =K' \unit_{\RR_+}  + K(0) \delta_0 \ .
$$

So we have, in the sense of distributions:
\begin{align*}
\frac{D}{D t} K_{\ell ,A} &=K_{\ell -1, A}+K_{\ell,A}(0, \sigma) \, \delta_0 \\
\frac{D^\ell}{D t^\ell} (K_{\ell , A}) &= K_{0,A} + K_{1,A} (0, \sigma) \, \delta_0 +
K_{2,A} (0, \sigma) \, \delta'_0 +\ldots
+ K_{\ell,A} (0, \sigma) \, \delta_0^{(\ell -1)}
\end{align*}
Using this, we get
\begin{equation} \label{eqn:variation}
e^{\sigma t}\frac{D^{d'}}{Dt^{d'}} K_{d' , A} (t,\sigma) = e^{\sigma t} K_{0,A} (t,\sigma) +
 e^{\sigma t} \sum_{\ell=1}^{d'} K_{\ell,A} (0, \sigma) \, \delta_0^{(\ell-1)}\, .
\end{equation}
From this last expression we see that away from $0$ the distributional limit, that we know to
exist, is independent of $\sigma$ and $d'$, since
$e^{\sigma t} K_{0,A} (t,\sigma)= \sum_{\rho\in A} e^{\rho t}$ is independent of
$\sigma$, and the other summand
is supported at zero.  
\end{proof}

In section \ref{parameter} we study in more detail the parameter dependence at $0$.

\section{Poisson-Newton formula}

On the half plane $\Re s >  \sigma_1$, $\log f(s)$ is well
defined taking the principal branch of the logarithm. Then we can
define the coefficients $(b_{\bk})$ by 
\begin{equation} \label{eqn:bn}
-\log f(s)=-\log \left ( 1+ \sum_{n\geq 1} a_n \ e^{-\lambda_n s}\right )
=\sum_{\bk \in \Lambda} b_{\bk} \, e^{-\langle \boldsymbol{\lambda} , \bk \rangle s}
 \ ,
 \end{equation}
where $\Lambda=\{ \bk=(k_n)_{n\geq 1} \, | \, k_n \in \NN, ||\bk||=\sum | k_n |<\infty, ||\bk|| \geq 1\}$,
and
$\langle \boldsymbol{\lambda} , \bk \rangle = \lambda_1k_1+\ldots + \lambda_{l}k_{l}$, where
$k_n=0$ for $n>l$.
Note that the coefficients $(b_{\bk})$ are polynomials on the $(a_n)$. More precisely, we have
\begin{equation} \label{eqn:bs}
 b_\bk= \frac{(-1)^{||\bk||}}{||\bk||} \, \frac{||\bk|| !}{\prod_j k_j!}\, \prod_j a_j^{k_j}\, .
\end{equation}

Note that if the $\lambda_n$ are $\QQ$-dependent then there are repetitions in
the exponents of (\ref{eqn:bn}).

\subsection{Hadamard interpolation}

\begin{lemma} \label{lem:G}
Consider a discrete set $\{\rho\} \subset \CC$  with the property that
$$
\sum_{\rho\neq 0} |n_\rho |  \, |\rho|^{-d}  < +\infty \, .
$$
Let $\sigma_1 \in \CC$. We have that
\begin{align*}
 G(s)&=\frac{n_{\sigma_1}}{s-\sigma_1}- \sum_{\rho \not=\sigma_1} n_\rho \left (
  \frac{1}{\rho-s} - \sum_{l=0}^{d-2} \frac{(s-\sigma_1)^l}{(\rho-\sigma_1)^{l+1}} \right ) \\
&=\frac{n_{\sigma_1}}{s-\sigma_1}+\sum_{\rho  \not=\sigma_1} n_\rho
  \frac{(s-\sigma_1)^{d-1}}{(\rho-\sigma_1)^{d-1}} \frac{1}{s-\rho} \, .
 \end{align*}

is a meromorphic function in $\CC$,
and has a simple pole with residue $n_\rho$ at each $\rho$.
\end{lemma}

\begin{proof} We start by noting that
 $$
  \frac{1}{\rho-s} - \sum_{l=0}^{d-2} \frac{(s-\sigma_1)^l}{(\rho-\sigma_1)^{l+1}}
   =\frac{(s-\sigma_1)^{d-1}}{(\rho-\sigma_1)^{d-1}} \frac{1}{\rho-s}\ .
 $$
Consider a disk $D(0,R)\subset \CC$, and split $G(s)=G_1(s)+G_2(s)$, where
$G_1(s)$ corresponds to the sum of those $\rho\in D(0,R)$, and $G_2(s)$ to the
sum over the remaining $\rho$'s. 

Now, for $s\in D(0,R/2)$, we have
 \begin{equation}\label{eqn:estimate}
 | G_2(s) |\leq C |s-\sigma_1|^d 
 \sum_\rho  |n_\rho| \, |\rho|^{-d} < \infty \ ,
 \end{equation}
thus we get the absolute and uniform convergence of the series in $D(0,R/2)$. 
As $G_1(s)$ has simple poles at $\rho$ with residues $n_\rho$, in $D(0,R/2)$,
we get the required properties for $G(s)$ in $D(0,R/2)$.
This happens for every $R>0$, thereby the result.

\end{proof}

Now we can define the Hadamard interpolation associated to the divisor $\sum n_\rho\,\rho$. We define it
up to a multiplicative constant which is irrelevant when we consider its logarithmic derivative as we will do.

\begin{definition}
We define the Hadamard interpolation as
$$
f_{H}(s)=\exp \left (\int G(s) \ ds \right ).
$$
The divisor of $f_{H}$ is $\Div f_{H}=\sum n_\rho\,\rho$.
\end{definition}

\begin{definition}
Consider a meromorphic function $f$ with divisor $\sum n_\rho\, \rho$. Fix $\sigma_1$ as above.
The discrepancy of $f$ is defined as the difference of the logarithmic derivatives
$$
P_f=\frac{f_{H}'}{f_{H}}-\frac{f'}{f} =G-\frac{f'}{f}
$$
\end{definition}

\begin{lemma}\label{lem:discrepancy}
 The discrepancy $P_f$ is a polynomial of degree $\leq g-1$.
\end{lemma}

\begin{proof}We recall the Hadamard factorization
$$
f(s)=(s-\sigma_1)^{n_{\sigma_1}} e^{Q_{f,\sigma_1}(s)}\prod_{\rho \not=\sigma_1} \left [E_{d-1}\left (
\frac{s-\sigma_1}{\rho-\sigma_1}\right )\right ]^{n_\rho} \ ,
$$
where $Q_{f,\sigma_1}$ is a polynomial of degree $\leq g$  which is uniquely
defined up to the addition of an
integer multiple of $2\pi i$. The product is
absolutely convergent because of the definition of the convergence exponent $d$. Taking the logarithmic derivative, we have that
$$
\frac{f'}{f}=Q_{f,\sigma_1}'+G(s) \ ,
$$
so $P_{f} =-Q_{f,\sigma_1}'$ and the result follows.
\end{proof}

\subsection{Poisson-Newton formula.}
The main result is the following Poisson-Newton formula associated to the Dirichlet series $f$.
Consider its polynomial discrepancy $P_f$ of degree $\leq g-1$,
$$
P_f(s)=c_0  +c_1 s +\ldots +c_{g-1} s^{g-1} \ .
$$
The inverse Laplace transform of $P_f$ is the distribution supported at $\{ 0\}$
$$
\cL^{-1} (P_f)=c_0 \delta_0 +c_1 \delta'_0 +\ldots +c_{g-1} \delta_0^{(g-1)}  \ .
$$
\begin{theorem}  \label{thm:main}
 As distributions in $\RR$ we have
 $$
  W (f)= \sum_{k=0}^{g-1} c_k \delta_0^{(k)} + \sum_{\bk \in \Lambda} \langle \lambda , \bk
\rangle \, b_{\bk} \ \delta_{\langle \boldsymbol{\lambda} ,\bk\rangle } \, .
 $$
\end{theorem}

The structure at $0$ of $W(f)$ depends on the function $f$ and its comparison with the Hadamard interpolation.
The structure out of $0$ only depends on the divisor and is independent of parameter choices. 
In some sense it is the most canonical part. Sometimes we refer to the ``full Poisson-Newton formula'' the 
one of the main theorem with the structure at $0$.

\begin{corollary}\label{cor:thm:main}
 As distributions on $\RR^*_+$ we have
 $$
  W (f)|_{\RR_+^*}= \sum_{\bk \in \Lambda} \langle \boldsymbol{\lambda} , \bk
\rangle \, b_{\bk} \ \delta_{\langle \boldsymbol\lambda ,\bk\rangle } \, .
 $$
\end{corollary}

\begin{proof}
We prove the theorem by taking the right-sided Laplace transform of $W(f)$  
(we use the interval $[-1,\infty)$):
 \begin{align*}
  \la W(f) , e^{-st} \ra_{[-1,\infty)}
  &= \left \la \frac{D^{d}}{Dt^d} F_{d}(t) , e^{(\sigma_1-s)t} \right \ra_{[-1,\infty)} \\
 &=  \int_0^\infty (-1)^d (K_d(t)-K_d(0)) \frac{d^d}{dt^d} e^{(\sigma_1-s)t} dt  \\
  &= n_{\sigma_1} (-1)^d \frac{(\sigma_1 -s)^d}{d!}\int_0^{+\infty} t^d e^{(\sigma_1-s)t} \, dt + \\
& +\sum_{\rho} \frac{n_\rho}{(\rho-\sigma_1)^d} (-1)^d (\sigma_1-s)^d 
  \left( \int_0^{+\infty} e^{(\rho-\sigma_1)t}e^{(\sigma_1-s)t} dt  
   - \int_0^{+\infty} e^{(\sigma_1-s)t} dt  \right) \\
  &= \frac{n_{\sigma_1}}{s-\sigma_1}-\sum_{\rho} n_\rho \frac{ (s- \sigma_1)^d}{(\rho-\sigma_1)^d} \left( \frac{1}{\rho-s} - \frac{1}{\sigma_1-s} \right) \\
   &= \frac{n_{\sigma_1}}{s-\sigma_1}-\sum_{\rho} n_\rho \frac{ (s- \sigma_1)^{d-1}}{(\rho-\sigma_1)^{d-1}}  \frac{1}{\rho-s} \\
  &=  G(s).
  \end{align*}
On the other hand, consider the distribution
  $$
   V=\sum_{\bk} \langle \boldsymbol{\lambda} , \bk\rangle \, b_{\bk} \ \delta_{\langle \boldsymbol{\lambda} , \bk\rangle}.
   $$
Its Laplace transform is
 \begin{align*}
   \la V, e^{-ts}\ra_{[-1,\infty)}
   &= \left \la \sum_{\bk} \langle \boldsymbol{\lambda} , \bk\rangle \, b_{\bk} \ \delta_{\langle \boldsymbol{\lambda} , \bk\rangle},
e^{-ts} \right \ra_{[-1,\infty)} \\
   &=\sum_{\bk} \langle \boldsymbol{\lambda} , \bk\rangle \, b_{\bk} e^{-\langle \boldsymbol{\lambda} , \bk \rangle s} \\
   &=- (-\log f(s))' \\
   &= \frac{f'(s)}{f(s)} \\
   &= G(s)-P_f(s) \\
   &= \la W(f) , e^{-st} \ra_{[-1,\infty)} -P_f (s)\, .
  \end{align*}
By uniqueness of the Laplace transform for distributions
(see \cite{Z}, Theorem 8.3-1, p.225), we have
  $$
  W(f)= V + \cL^{-1}(P_f),
  $$
where $\cL^{-1}(P_f)$ is the inverse Laplace transform of the polynomial $P_f$. This is a
distribution supported at $\{0\}$. Hence we get the theorem and the corollary
  $$
  W(f)|_{\RR_+^*}= V \, .
  $$
\end{proof}

Just inspecting the order of the distributions appearing in both sides of the Poisson-Newton formula, we get 
an interesting corollary for Dirichlet series. We know that $d\leq g+1$. In fact we do have equality $d=g+1$.

\begin{corollary} \label{cor:genus}
 For a meromorphic Dirichlet series we have 
$$
d=g+1=\deg Q_f +1=\deg P_f +2 \ .
$$
\end{corollary}

\begin{proof}
We inspect the order of the distributions in the Poisson-Newton-formula.
 We recall that $W(f)$ is, as distribution, the $d$-th derivative of a continuous function. But $\delta_0^{(l)}$ is not the $d$-th derivative of a continuous function for $l\geq d-1$. Thus 
$\deg P_f \leq  d-2$ so $\deg Q_f \leq d-1$ and 
$g\leq d -1$, hence $g+1=d$.
\end{proof}


It is clear that $d\geq 1$ for a meromorphic Dirichlet series, but we have in fact $d\geq 2$.

\begin{corollary} \label{cor:exp}
 For a meromorphic Dirichlet series we have a convergence exponent at least $2$ and order at least $1$:
$$
d \geq 2
$$
and
$$
o \geq 1 \ .
$$ 
\end{corollary}

\begin{proof}
As before we inspect the order of the distributions in the Poisson-Newton-formula. The right hand side contains Dirac
distributions at the frequencies, hence it is at least a second derivative of a continuous function. In the left 
hand side we have $W(f)$ that is the $d$-th derivative of a continuous function. This gives $d\geq 2$.

Also we know that $d\leq o +1$, hence $o\geq 1$.
\end{proof}

\subsection{Symmetric Poisson-Newton formula.}

Let $f(s)$ be a Dirichlet series with exponent of convergence $d$, and fix $\sigma$ as before. We have
defined a distribution $W(f,\sigma)(t)=\left( \sum n_\rho e^{\rho t}\right) \unit_{\RR_+}$ 
supported on $\RR_+$. If we make the change
of variables $t\mapsto -t$, we have the distribution $W(f,\sigma)(-t)=\left(
\sum n_\rho e^{-\rho t}\right) \unit_{\RR_-}$, which is formally
defined as
 $$
  (-1)^d e^{-\sigma t} \frac{D^d}{Dt^d} \left( (K_d(-t)-K_d(0)) \mathbf{1}_{\RR_-} \right).
 $$
 This is independent of $\sigma$ on $\RR^*_-$ and has a contribution at zero dependent on the parameter.
 
 The sum 
  $$
  \widehat{W}(f,\sigma)=W(f,\sigma)(t)+  W(f,\sigma)(-t)
  $$
 is a distribution on $\RR$, whose only dependence on $\sigma$ is at zero, and which formally it is equal to
 $\sum_\rho n_\rho e^{\rho |t|}$, $t\in \RR$.

 \begin{theorem}\label{thm:symmetric}
 For a Dirichlet series $f$, we have in $\RR$,
$$
\widehat{W}(f,\sigma)(t)= 2 \sum_{l=0}^{\frac{g-1}{2}} c_{2l} \, \delta_0^{(2l)} + 
\sum_{\bk \in \Lambda \cup (-\Lambda )} \langle \boldsymbol{\lambda} , |\bk |
\rangle \, b_{|\bk |} \ \delta_{\langle \boldsymbol{\lambda} ,\bk\rangle } \, .
 $$
\end{theorem}

\begin{proof}
The Poisson-Newton formula for $f$ is 
$$
W(f,\sigma)(t)=\sum_{\bk \in \Lambda} \langle \boldsymbol{\lambda} , \bk  \rangle 
b_{\bk} \, \delta_{\langle \boldsymbol{\lambda} ,\bk\rangle }+\sum_{l=0}^{g-1} c_{l} \,  \delta_0^{(l)} \, ,
$$
Making the change of variables $t\mapsto -t$, we have 
$$
W(f,\sigma)(-t)  = \sum_{\bk \in -\Lambda} \langle \boldsymbol{\lambda} , |\bk | \rangle 
b_{|\bk |} \, \delta_{\langle \boldsymbol{\lambda} ,\bk\rangle }+\sum_{l=0}^{g-1} (-1)^l c_{l} \, 
\delta_0^{(l)} \, ,
$$
where $|\bk|=-\bk$, for $\bk \in -\Lambda$. 

Adding the two formulas, we get the symmetric formula stated in the theorem for
$$
\widehat{W}(f,\sigma)(t)= W(f,\sigma)(t) +W(f,\sigma)(-t)=\sum_\rho n_\rho e^{\rho |t|}  \ .
$$
\end{proof}

\bigskip

Consider a Dirichlet series $f(s)=1+\sum a_n e^{\lambda_n s}$ and let 
 $$
 \bar f (s)=\overline{f(\bar s)}=1+\sum \bar a_n e^{\lambda_n s}
 $$
be its conjugate. Then $\bar f$ is a Dirichlet series whose 
zeros are the $\{\bar \rho\}$ and $n_{\bar \rho} =n_\rho$. Also $b_\bk (\bar f)=\overline{b_\bk (f)}$.
The Poisson-Newton formula for $\bar f$ is 
$$
W(\bar f,\bar \sigma)(t)=\sum_{\bk \in \Lambda} \langle \boldsymbol{\lambda} , {\bk}  \rangle 
\overline{b_\bk} \, \delta_{\langle \boldsymbol{\lambda} ,{\bk}\rangle }+\sum_{l=0}^{g-1} 
\bar c_{l} \,  \delta_0^{(l)} \, ,
$$

\begin{corollary}\label{real-analytic:symmetric}
 For a real analytic Dirichlet series $f$, that is $\bar f (s)=f(s)$, we have that
 for $\sigma\in \RR$, the numbers $c_l$ and $b_\bk$ are real.
 
 The converse also holds.
\end{corollary}

The last point is due to the fact that the association $f\mapsto W(f)$ is one-to-one, as its inverse is the Laplace transform.

\subsection{Poisson-Newton formula with parameters.}

Observing that the space of Dirichlet series is invariant by
the change of variables
$s\mapsto \alpha s +\beta$, with $\alpha >0$ and $\beta \in \CC$, we get a parameter version
of the main theorem.

\begin{corollary} \label{cor:3.7} 
Let $\alpha >0$ and $\beta \in \CC$.
 As distributions on $\RR^*_+$ we have
 $$
  e^{-\frac{\beta}{\alpha} t} \  W (f)(t/\alpha )|_{\RR_+^*}= \sum_{\bk \in \Lambda} \alpha
\langle \boldsymbol{\lambda} , \bk
\rangle \, e^{-\langle \boldsymbol{\lambda} , \bk \rangle \beta} \, b_{\bk} \ \delta_{\alpha \langle \boldsymbol{\lambda} ,\bk\rangle } \, .
 $$
\end{corollary}

\begin{proof}
This results by applying Corollary \ref{cor:thm:main} to $g(s)=f(\alpha s +\beta)$, which is
a Dirichlet series for $\alpha >0$,
$$
g(s)=1 + \sum_{n\geq 1} a_n e^{-\lambda_n \beta} e^{-(\lambda _n \alpha) s} \ .
$$
 The zeros of $g$ are the numbers $\left(\frac{\rho- \beta }{\alpha}\right)$, and $b_\bk$ is changed to
$e^{-\langle \boldsymbol{\lambda} , \bk \rangle \beta} b_\bk$.

\end{proof}

We can also give a parameter version of the full Poisson-Newton formula.

\begin{corollary} Let $\alpha >0$ and $\beta \in \CC$.
 As distributions in $\RR$ we have
 $$
  e^{-\frac{\beta}{\alpha} t} \  W (f)(t/\alpha )=  
\sum_{k=0}^{g-1}  c_k (\sigma_1' ) \delta_0^{(k)} +
\sum_{\bk \in \Lambda} \alpha
\langle \boldsymbol{\lambda} , \bk
\rangle \, e^{-\langle \boldsymbol{\lambda} , \bk \rangle \beta} \, b_{\bk} \ \delta_{\alpha \langle \boldsymbol{\lambda} ,\bk\rangle } \, ,
 $$
where the $c_k (\sigma_1' )$ are the coefficients of the discrepancy 
polynomial for $\sigma=\sigma_1'=\frac{\sigma_1-\Re \beta}{\alpha}$.
\end{corollary}

\begin{proof}
Just observe that $\sigma_1$ becomes $\sigma_1'=\frac{\sigma_1-\Re \beta}{\alpha}$ for $g(s)=f(\alpha s +\beta )$.
\end{proof}

\subsection{Symmetric Poisson-Newton formula with parameters.}

The space of real analytic Dirichlet series is invariant by
the change of variables
$s\mapsto \alpha s +\beta$, with $\alpha >0$ and $\beta \in \RR$, then we get a parameter version
of the symmetric Poisson-Newton formula of the previous section.

\begin{theorem}\label{thm:symmetric_parameters} Let $\alpha >0$ and $\beta \in \RR$.
 For a real analytic Dirichlet series $f$, that is $\bar f (s)=\overline{f(\bar s)}=f(s)$, we have in $\RR$
\begin{align*}
&e^{-\frac{\beta}{\alpha} |t|} (W(f)(t)+W(f)(-t)) =\sum_\rho n_\rho e^{(\rho-\beta) |t|/\alpha}= \\
&= 2 \sum_{l=0}^{\frac{g-1}{2}}  c_{2l}(\sigma_1')
 \, \delta_0^{(2l)} + \sum_{\bk \in \Lambda \cup (-\Lambda )} \alpha \, \langle \boldsymbol{\lambda} , |\bk |
\rangle e^{-\langle \boldsymbol{\lambda} , |\bk |\rangle \beta}   \, b_{|\bk |} \ \delta_{\alpha \langle \boldsymbol{\lambda} ,\bk\rangle } \, ,
\end{align*}
with $\sigma_1'=\frac{\sigma_1-\beta}{\alpha}$.
\end{theorem}

In particular, for $\alpha =1$ and $g=1$ that we use in the applications, we get
\begin{corollary}\label{cor:symmetric_parameters} Let $\beta \in \RR$.
 For a real analytic Dirichlet series $f$, that is $\bar f (s)=\overline{f(\bar s)}=f(s)$, of genus $g=1$ 
we have in $\RR$
$$
\sum_\rho n_\rho e^{(\rho-\beta) |t|}= 2  c_0(\sigma_1-\beta) \, \delta_0 + \sum_{\bk \in \Lambda \cup (-\Lambda )}  \, \langle \boldsymbol{\lambda} , |\bk |
\rangle e^{-\langle \boldsymbol{\lambda} , |\bk |\rangle \beta}   \, b_{|\bk |} \ \delta_{ \langle \boldsymbol{\lambda} ,\bk\rangle } \, .
$$
\end{corollary}

\subsection{General Poisson-Newton formula.}

We can observe that in the proof of the Poisson-Newton formula we barely used the Dirichlet character of $f$.
The construction of $G$, the discrepancy polynomial $P_f$, and the Hadamard interpolation $f_{H}$
can be carried out in general for any meromorphic
function of finite order with its divisor contained  in a left half plane.
The existence of the Newton-Cramer distribution $W(f)$ also holds
in this generality.
The choice of $\sigma_1$ 
is only relevant for the structure at $0$ of the
Newton-Cramer distribution and for determining
the exponential decay of the test functions to which we can apply the distribution formula. Only in order
to compute the logarithmic derivative $f'/f$ we have used the Dirichlet series expansion. If we take over
the proof for a meromorphic funciton $f$ of finite order and with divisor contained in a
left half plane, we get
the following general theorem:

\begin{theorem}
 Let $f$ be a meromorphic function of finite order with convergence exponent $d$ and
its divisor contained in a left half plane. Let $W(f)$
be its Newton-Cramer distribution and $P_f(s)=c_0+c_1 s+\ldots +c_{g-1} s^{g-1}$ 
be the discrepancy polynomial.
We have in $\RR$,
$$
W(f)=\sum_{j=0}^{g-1} c_j \delta_0^{(j)} + \cL^{-1} (f'/f) \ ,
$$
or
$$
\cL (W(f))=\langle W(f) , e^{-st}\rangle_{\RR_+} =\frac{f'(s)}{f(s)} +P_f(s) \ .
$$
\end{theorem}

The inverse Laplace transform $\cL^{-1}(F)$ ia a well defined distribution of finite order 
when $F$ has polynomial growth on a half plane, which is the case of $F=f'/f$
when $f$ is of finite order and has divisor contained on a left half plane. It is defined as follows.
Take $m$ which is two units more than the order of growth of $F$, and define
 $$
 L(t)=\int_{-\infty}^{+\infty} \frac{F(c+iu)}{(c+iu)^m} e^{(c+iu)t} \, \frac{du}{2\pi} \, .
  $$
This is a continuous function, which vanishes on $\RR_-$. It is independent of the choice of $c$
(subject to $\Re c>\sigma_1$).
Then
$$
\cL^{-1}(F) (t) := \frac{D^m}{Dt^m} L(t),
$$
which is a distribution of order at most $m-2$.

More explicitly, for an appropriate test function $\varphi$, letting $\psi (t)=\varphi (t) e^{ct}$, we have
\begin{align*}
\langle \cL^{-1}(F) , \varphi \rangle & = \langle L(t), (-1)^m \varphi^{(m)}(t)\rangle \\
 &= \int_\RR 
\int_{-\infty}^{+\infty}\frac{F(c+iu)}{(c+iu)^m} (-1)^m \varphi^{(m)}(t) e^{(c+iu)t} \, \frac{du}{2\pi} dt \\
&= \int_{-\infty}^{+\infty} (-1)^m \frac{F(c+iu)}{(c+iu)^m} 
 \left( \int_\RR  \varphi^{(m)}(t) e^{(c+iu)t}  dt \right) \, \frac{du}{2\pi}  \\
 &= \int_{-\infty}^{+\infty} \frac{F(c+iu)}{(c+iu)^m} 
 \left( \int_\RR  (c+iu)^m \varphi(t) e^{(c+iu)t}  dt \right) \, \frac{du}{2\pi}  \\
 &=\int_{-\infty}^{+\infty} F(c+iu)
 \hat \psi(-u) \, \frac{du}{2\pi}  \, ,
\end{align*}
doing $m$ integrations by parts in the penultimate line.

We can give a symmetric version of the general Poisson-Newton formula: We  denote
  $$
  \widehat{W}(f,\sigma)=W(f,\sigma)(t)+  W(f,\sigma)(-t),
  $$
 which is a distribution on $\RR$.

 \begin{theorem}\label{thm:symmetric-for-general}
 We have in $\RR$
$$
\widehat{W}(f,\sigma)(t)= 2 \sum_{l=0}^{\frac{g-1}{2}} c_{2l} \, \delta_0^{(2l)} + 
 \big( \cL^{-1}(f'/f)(t) + \cL^{-1}(f'/f)(-t) \big).
$$
 \end{theorem}

\subsection{General symmetric Poisson-Newton formula with parameters}
Let $\alpha >0$, $\beta\in \CC$. Take $\sigma_1'=\frac{\sigma-\beta}{\alpha}$. 
We denote, as a slight abuse of notation, 
$$
 e^{-\frac{\beta}{\alpha}|t|} \widehat{W}(f,\sigma)(t)=
 e^{-\frac{\beta}{\alpha} t} {W}(f,\sigma)(t) + e^{\frac{\beta}{\alpha} t} W(f,\sigma)(-t)\, .
$$

 \begin{theorem}\label{thm:symmetric-parameters-general}
 We have in $\RR$
$$
 e^{-\frac{\beta}{\alpha} |t|} \widehat{W}(f,\sigma)(t/\alpha)= 
  2 \sum_{l=0}^{\frac{g-1}{2}} c_{2l}(\sigma_1') \, \delta_0^{(2l)} + 
  \big( e^{-\frac{\beta}{\alpha} t} \cL^{-1}(f'/f)(t/\alpha) + 
 e^{\frac{\beta}{\alpha} t} \cL^{-1}(f'/f)(-t/\alpha) \big).
$$
 \end{theorem}

 \begin{corollary}\label{cor:symmetric-parameters-general}
 For a real analytic function $f$ and $\beta\in \RR$ to the right of all zeroes of $f$, we have in $\RR$
$$
 e^{-\frac{\beta}{\alpha} |t|} \widehat{W}(f,\sigma)(t)= 
  2 \sum_{l=0}^{\frac{g-1}{2}} c_{2l}(\sigma_1') \, \delta_0^{(2l)} + 
  e^{-\frac{\beta}{\alpha} t} \cL^{-1}_\beta  \left( 2 \Re(f'/f)\right) (t/\alpha).
$$
 \end{corollary}

We prove this using the following result
 
\begin{lemma}
For a real analytic function $F$, and $\gamma$ to the right of the zeroes, we have
$$
  e^{-\gamma t} \cL^{-1}(F)(t) + 
 e^{\gamma t} \cL^{-1}(F)(-t) = 
  e^{-\gamma t} \cL^{-1}_\gamma  \left( 2 \Re F\right) (t)
  $$
\end{lemma}

\begin{proof}
 We have 
 $$
   e^{-c t} \cL^{-1}(F)(t) = \int_{-\infty}^{+\infty} F(c+iu ) e^{iut} \frac{du}{2\pi}
 $$
Analogously,
\begin{align*}  
     e^{c t} \cL^{-1}(F)(-t) &= \int_{-\infty}^{+\infty} F(c+iu ) e^{-iut} \frac{du}{2\pi} \\
     &= \int_{-\infty}^{+\infty} F(c-iv ) e^{ivt} \frac{dv}{2\pi} \\
     &= \int_{-\infty}^{+\infty} \overline{F(c+iv )} e^{ivt} \frac{dv}{2\pi} 
 \end{align*}
 Adding both, we get
  $$
   \int_{-\infty}^{+\infty} 2 \left( \Re F(c+iv ) \right) e^{ivt} \frac{dv}{2\pi}
  $$
 as required.
  \end{proof}

\subsection{Converse theorem.}

Conversely we show that any Poisson-Newton formula is associated to a Dirichlet series.

\begin{proposition}
 Let $\sum_\rho n_\rho \, \rho$ be a divisor with convergence exponent equal to $d$
 and contained in a left half plane $\Re \rho \leq \sigma_1$. Suppose that
a Poisson-Newton formula
 $$
 \left(\sum_\rho n_\rho e^{-\rho t}\right) \unit_{\RR_+} 
= \sum_{n\geq 1} b_n \delta_{\mu_n} + \sum_{j=0}^{d-2} c_j \delta_0^{(j)}
 $$
holds in $\RR$, in the sense of an equality of distributions,
acting on functions with fast enough exponential decay. We assume
$0<\mu_1 <\mu_2 <\ldots$, where $(\mu_n)$ are finitely many
or $\mu_n  \to +\infty$, and $\sum_{n\geq 1} |b_n| e^{-\mu_n \bar\sigma} <\infty$.
for some $\bar\sigma< \infty$.
Then there is a Dirichlet series $f$ meromorphic on
$\CC$, whose Poisson-Newton formula is the one given.
\end{proposition}

\begin{proof}
The condition 
$\sum |n_\rho|\, |\rho|^{-d}<\infty$ allows us to define
the function $G(s,\sigma)$ associated to the divisor, and the corresponding
Hadamard interpolation $f_{H}(s,\sigma)=\exp \left(\int G(s,\sigma)\, ds\right)$.
The function $f_{H}(s,\sigma)$ is meromorphic on $\CC$ and has divisor of
zeros and poles equal to $\sum n_\rho\, \rho$. For the function $f_{H}$, we have
defined a distribution $W(f_{H})$ supported in $\RR_+$, which equals
with $\sum n_\rho e^{-\rho t}$ in $\RR_+^*$ (the precise meaning of the
later is our definition of the distributional sense given in Section \ref{sec:2}).
Therefore,
 $$
 W(f_{H})= \left( \sum_\rho n_\rho e^{-\rho t} \right) \unit_{\RR_+} + \sum_{j=0}^{d-2} c'_j\delta_0^{(j)} \ .
 $$
By our hypothesis, we have the equality
 $$
 W(f_{H})= \sum_{n\geq 1} b_n \delta_{\mu_n}+ \sum_{j=0}^{d-2} c''_j\delta_0^{(j)}
 $$
as distributions paired against functions with fast enough exponential decay.
So for $\Re s \geq \bar\sigma$ (enlarging $\bar\sigma$ if necessary), we have
 \begin{align*}
 \frac{f'_{H}}{f_{H}} &=  G(s)= \langle W(f_{H}), e^{-ts} \rangle \\
&= \left \langle \sum_{n\geq 1} b_n \delta_{\mu_n}+ \sum_{j=0}^{d-2} c''_j\delta_0^{(j)} , e^{-ts} \right \rangle \\
&= F(s) +P(s),
 \end{align*}
where $P(s)$ is a polynomial of degree at most $d-2$, and
  $$
 F(s)=\left \langle \sum_{n\geq 1} b_n \delta_{\mu_n}, e^{-ts} \right \rangle =  \sum_{n\geq 1} b_n e^{-\mu_n s} \, .
 $$
Note that the assumptions mean that this can be paired against $e^{-st}$ for $\Re s\geq \bar\sigma$
and the last sum is uniformly convergent. 
So for $\Re s\geq \bar\sigma$, we have
 $$
 f(s):=e^{\int F(s)} = e^{-\int P(s)} f_{H} (s)\, .
 $$
Then $f(s)$ is a meromorphic function on $\CC$, and in a right half-plane, we have
 $$
 f(s)= \exp \left ( \sum b_n e^{-\mu_n s} \right ) = 1+ \sum_{n\geq 1} a_n e^{-\lambda_n s}\, ,
 $$
which is a Dirichlet series.
\end{proof}

\section{Parameter dependence} \label{parameter}

We analyze the dependence of $W(f,d',\sigma)$ on $d'$ and $\sigma$,
which is concentrated at $0$. We write
 $$
 R_A(\sigma,d') = \sum_{\ell=1}^{d'} e^{\sigma t} K_{\ell,A} (0, \sigma) \, \delta_0^{(\ell-1)}
 $$
for the remainder in (\ref{eqn:variation}). Note that this is not convergent, as the first
summand in the right hand side (\ref{eqn:variation}) is also not convergent. But
 $$
 R_A(\sigma_1,d'_1) -R_A(\sigma_2,d'_2) = W_A(f,\sigma_1,d'_1)- W_A(f,\sigma_2,d'_2)
 $$
is convergent. The dependence on $d'\geq d$ is easy to describe,
since the $K_{\ell,A}(0,\sigma)$
do converge for $\ell\geq d$.

\begin{lemma}Let $d'=d+a$, $a\geq 1$. Then
 \begin{align*}
 &W(f,\sigma, d') =W(f,\sigma,d) + \sum_{j=1}^a K_{d+j}(0,\sigma) e^{\sigma t}\,  \delta_0^{(d-1+j)} \\
&=W(f,\sigma,d) + \sum_{k=0}^{d-1+a} \left ( \sum_{\max (1, k-d+1) \leq j \leq a} \binom{d-1+j}{k} (-\sigma)^{d-1+j-k} K_{d+j} (0,\sigma )\right ) \, \delta_0^{(k)}
 \end{align*}
In particular $W(f,\sigma,d)$ is obtained from $W(f,\sigma,d')$ by removing the higher order derivatives
of $\delta_0$ strictly larger than $d-1$ and a convergent part in lower order derivatives.
\end{lemma}

\begin{proof}
 Just note that for $n\geq 0$ and $\sigma \in \CC$,
$$
e^{\sigma t} \, \delta_0^{(n)} =\sum_{k=0}^n \binom{n}{k} (-\sigma )^{n-k} \delta_0^{(k)} \ .
$$
\end{proof}

So from now on, we shall restrict to $d'=d$ and write $W(f,\sigma)=W(f,\sigma,d)$.
We write
\begin{align*}
R_A(\sigma)= W_A(f, \sigma)-W_A(f, \sigma_1) &=\sum_{l=0}^{d-1}  (K_{l+1, A}(0, \sigma ) e^{\sigma t} -K_{l+1,A}(0, \sigma_1 ) e^{\sigma_1 t})\delta_0^{(l)} \\
&=\sum_{k=0}^{d-1}  r_{k, A}(\sigma )  \delta_0^{(k)}
\end{align*}
with for $k=0, 1, \ldots , d-1$
\begin{align*}
c_{k,A} (\sigma )&=\sum_{l=k}^{d-1} \binom{l}{k}  (-\sigma)^{l-k} K_{l+1,A}(0,\sigma ) \\
r_{k,A} (\sigma )&= c_{k,A} (\sigma )-c_{k,A} (\sigma_1 ) \ ,
\end{align*}
and
$$
R(\sigma)=\lim_A R_A(\sigma) =\sum_{k=0}^{d-1}  r_{k}(\sigma )  \delta_0^{(k)} \ ,
$$
with
$$
r_{k}=\lim_A r_{k,A} \ .
$$

Note that the last coefficient for $k=d-1$ is clearly convergent and the following holds:

\begin{proposition}We have that $r_{d-1} (\sigma)=K_d(0,\sigma)-K_d(0,\sigma_1)=
 \sum_\rho \frac{n_\rho}{(\rho-\sigma)^d} -K_d(0,\sigma_1)$ is a meromorphic function on $\CC$
with poles of order $d$ at the $\rho$'s.
\end{proposition}

For $k\leq d-2$, we do not have such an explicit description, but still
a similar result holds.

\begin{proposition}
 For $k=0, 1,\ldots , d-2$, the coefficient $r_k(\sigma)$ is a meromorphic function
on $\sigma\in \CC$ with poles of order $d$ at the $\rho$'s.
\end{proposition}

\begin{proof}
This results from the parameter dependence of the Hadamard interpolation.
We have defined earlier
$$
G(s,\sigma) = -\sum_{\rho} n_\rho \left (
  \frac{1}{\rho-s} - \sum_{l=0}^{d-2} \frac{(s-\sigma )^l}{(\rho-\sigma )^{l+1}} \right )
   =\sum_{\rho} n_\rho
  \frac{(s-\sigma )^{d-1}}{(\rho-\sigma )^{d-1}} \frac{1}{s-\rho} \ .
$$
Then we write
$$
f_{H}(s,\sigma) =\exp \left (\int G(s,\sigma ) \, ds \right )  ,
$$
and
$$
P_f (s,\sigma )=\frac{f'_{H}(s,\sigma)}{f_{H}(s,\sigma)}-\frac{f' (s)}{f (s)}=G(s, \sigma)-\frac{f' (s)}{f (s)}
$$
and
$$
P_f(s, \sigma)=c_0 (\sigma )  +c_1 (\sigma) s +\ldots +c_{d-2}(\sigma) s^{d-2} \ .
$$

Assume that $0$ is not in the support of the divisor. Then observe that 
 $$
 c_j(\sigma) = \frac{1}{j!} \, P_f^{(j)}(0,\sigma),
 $$
taking the derivative with respect to $s$, $j=0,1,\ldots, d-2$. As $G(s,\sigma)$ has poles of order $d-1$ at
$\rho$ (as a function on $\sigma$), the same happens for $c_j(\sigma)$.

\end{proof}

\subsection{Newton-Cramer functions associated to divisors}

Consider a divisor $D=\sum n_\rho\, \rho$ supported in a left half-plane and
of convergence exponent $d$. This means that $\Re \rho\leq \sigma_1$ for some $\sigma_1 \in \RR$ and
$\sum_{\rho\neq 0} |n_\rho| \, |\rho|^{-d}<\infty$.
 Then, for any $\sigma\in \CC-\{\rho\}$ and $d'\geq d$, we have an associated
 distribution
  $$
  W(D,\sigma, d')
  $$
constructed as in Section \ref{sec:2}. 
All these distributions are supported in $\RR_+$, and they are independent of $\sigma$ and $d'$ in
$\RR_*^+$. Moreover, we denote $W(D,\sigma)=W(D,\sigma,d)$.

The distributions $W(D,\sigma)$ are Laplace transformable, i.e., they can be paired against
$e^{-st}$, for $\Re s> \sigma_0$, for some $\sigma_0$ depending on $D$ and $\sigma$.
Define
 $$
 g_{D}(s,\sigma)= \la W(D,\sigma), e^{-st}\ra_{\RR}
 $$
and
 $$
 f_{D}(s,\sigma)=\exp\left( - \int g_{D}(s,\sigma) \, ds\right).
 $$

\begin{proposition}
 All $g_{D}(s,\sigma)$, for different $\sigma$, are equal up to the addition of
 a polynomial of degree $d-2$ in $s$. They have simple poles exactly at the
 $\rho$, with residues $n_\rho$.

 The functions $f_{D}(s,\sigma)$ are meromorphic on $\CC$. They have the same divisor
and they differ by a Weierstrass exponential factor of order
 at most $d-1$.
\end{proposition}

\section{Applications.}

\subsection{Classical Poisson formula.}

Consider for $\lambda=\lambda_1 >0$ the entire function
$$
f(s)=1-e^{-\lambda s} \, .
$$
This function has order $1$, exponent of convergence $d=2$, genus $g=1$, and its zeros are
$$
\rho_k= i\frac{2\pi}{\lambda} k \, ,
$$
for $k\in \ZZ$. Also we have
$$
-\log f(s) =-\log \left ( 1- e^{-\lambda s}\right )=\sum_{k=1}^{+\infty } \frac{1}{k} e^{-\lambda k s} \ ,
$$
i.e., with our notation $\Lambda=\NN^*$, $b_{k}=1/k$.

Therefore the Poisson-Newton formula in $\RR_+^*$ gives
$$
\left (\sum_{k\in \ZZ} e^{i\frac{2\pi}{\lambda} k t} \right )_{\big|\RR_+^*}
=\lambda \sum_{k\geq 1}  \delta_{\lambda k} \, .
$$
The distribution on the left side, when considered without restricting to $\RR_+^*$, is even. It follows then
$$
\left (\sum_{k\in \ZZ} e^{i\frac{2\pi}{\lambda} k t} \right )_{\big|\RR^*}
=\lambda \sum_{k\in \ZZ^*}  \delta_{\lambda k} \, .
$$
Now the distribution on the left side, without restricting to $\RR^*$, is $\lambda$-periodic on $\RR$. 
So we get
the full classical Poisson formula, identifying what we get at $0$,
$$
\sum_{k\in \ZZ} e^{i\frac{2\pi}{\lambda} k t}  =\lambda \sum_{k\in \ZZ}  \delta_{\lambda k} \, .
$$

If we start with
$$
f(s)=1-a e^{-\lambda s}\, ,
$$
for some $a\in \CC^*$, we get also the classical Poisson formula. We have for $k\in \ZZ$,
$$
\rho_k= \frac{1}{\lambda} \log a +i\frac{2\pi}{\lambda} k \, ,
$$
and
$$
b_{k} = \frac{1}{k} a^k \, .
$$
Therefore
$$
\left (a^{t/\lambda} \sum_{k\in \ZZ} e^{i\frac{2\pi}{\lambda} k t} \right )_{\big|\RR_+^*} =
\lambda \sum_{k\geq 1}  a^k\delta_{\lambda k}  =
\lambda a^{t/\lambda} \sum_{k\geq 1}  \delta_{\lambda k}\, ,
$$
which gives the same formula as before.

\subsection*{Application of the symmetric Poisson-Newton formula.}

As expected from the symmetric form of the classical Poisson formula, one can recover the formula in a more
direct way from the symmetric Poisson-Newton formula.

It is also interesting to clarify the structure of the
Newton-Cramer distribution at $0$. It helps to understand why the 
Dirac $\delta_0$ appearing in the right side of the
classical Poisson formula is of a different nature than the other $\delta_{\lambda k}$ for $k\not= 0$, 
something that was intuitively suspected from the analogy with trace formulas
(see a comment on this in \cite{CV}, p.2).

In order to use the symmetric Poisson-Newton formula we compute the discrepancy polynomial $P_f$ for
$f(s)=1-e^{-\lambda s}$. We have that $\sigma_1=0$ is a zero. From the classical Hadamard factorization
$$
\sinh (\pi s)=\pi s \prod_{k\in \ZZ^*} \left (1-\frac{s}{ik}\right ) e^{\frac{s}{ik}} \ ,
$$
we get the Hadamard factorization for $f$,
$$
f(s)=2 e^{-\lambda s /2} \sinh(\lambda s/2)=s \lambda e^{-\lambda s/2} \prod_{k\in \ZZ^*} \left (1-\frac{s}{\rho_k}\right ) e^{\frac{s}{\rho_k}} \ .
$$
Note that this is equivalent to
\begin{align*}
G(s) &=\frac{1}{s} - \sum_{k\in \ZZ^*} \left (\frac{1}{\rho_k-s}-\frac{1}{\rho_k} \right )  \\
&=\frac{\lambda/2}{\tanh\left (\lambda s/2\right )} \ .
\end{align*}

Thus $Q_f(s)=(\log \lambda +2\pi i n)-\frac{\lambda}{2} s$, with $n\in \ZZ$, and
$$
P_f(s)=-Q_f'(s)=c_0=\frac{\lambda}{2} \ .
$$

Therefore we can apply the symmetric Poisson-Newton formula (Theorem \ref{thm:symmetric}) and we get
\begin{align*}
\sum_{k\in \ZZ} e^{i\frac{2\pi}{\lambda} k|t|} &= 2 c_0 \delta_0 +
\lambda \sum_{k\in \ZZ^*} \delta_{\lambda k} \\
&= \lambda \delta_0 +
\lambda \sum_{k\in \ZZ^*} \delta_{\lambda k} \\
&=\lambda \sum_{k\in \ZZ} \delta_{\lambda k} \ .
\end{align*}

We finally observe that
$$
\sum_{k\in \ZZ} e^{i\frac{2\pi}{\lambda}k |t|} =\sum_{k\in \ZZ} e^{i\frac{2\pi}{\lambda}k t} \ ,
$$
because we can reorder freely a converging (in the distribution sense) infinite series of distributions
\begin{align*}
\sum_{k\in \ZZ} e^{i\frac{2\pi}{\lambda} k|t|} &= 1+2\sum_{k=1}^{+\infty} \cos \left (\frac{2\pi}{\lambda} k |t|
\right ) \\
&= 1+2\sum_{k=1}^{+\infty} \cos \left (\frac{2\pi}{\lambda} k t
\right ) \\
&=\sum_{k\in \ZZ} e^{i\frac{2\pi}{\lambda}k t }\ .
\end{align*}

\subsection{Newton formulas.}

We show in this section how the Poisson-Newton formula is a generalization to Dirichlet series of Newton
formulas which express Newton sums of roots of a polynomial equation in terms of its coefficients
(or  elementary symmetric functions).

Let $P(z)=z^n+a_1 z^{n-1}+\ldots +a_n $ be a polynomial of degree $n\geq 1$ with zeros $\alpha_1, \ldots, \alpha_n$
repeated according to their multiplicity. 
For each integer $m\geq 1$, the Newton sums of the roots are the symmetric functions
$$
S_m=\sum_{j=1}^n \alpha_j^m \ .
$$
From the fundamental theorem on symmetric functions, these Newton sums can be expressed polynomially with
integer coefficients in terms of
elementary symmetric functions, i.e., in terms of the coefficients of $P$. These are the Newton formulas. For instance, if for $k\geq 1$
$$
\Sigma_k =\sum_{1\leq i_1<\ldots < i_k \leq n} \alpha_{i_1}\ldots \alpha_{i_k} = (-1)^k a_k \ ,
$$
then we have
\begin{align*}
S_1&= \Sigma_1 \\
S_2&= \Sigma_1^2 - 2 \Sigma_2 \\
S_3&= \Sigma_1^3 -3\Sigma_2 \Sigma_1 +3\Sigma_3 \\
S_4&= \Sigma_1^4 -4\Sigma_2\Sigma_1^2 +4 \Sigma_3 \Sigma_1 +2\Sigma_2^2 -4\Sigma_4 \\
&\vdots
\end{align*}

We recover
them applying the Poisson-Newton formula to the finite Dirichlet series
$$
f(s)=e^{-\lambda ns }P(e^{\lambda s}) =1+a_1 e^{-\lambda s} +\ldots + a_n e^{-\lambda n s} \ .
$$
The zeros of $f$ are the $(\rho_{j,k})$ with $j=1,\ldots , n$, $k\in \ZZ$, and
$$
e^{\rho_{j, k}}= \alpha_j^{1/\lambda} e^{\frac{2\pi i}{\lambda}k} \ .
$$
Thus, using the classical Poisson formula, its Newton-Cramer distribution can be computed in $\RR$ as
\begin{align*}
 \sum_\rho e^{\rho t} &= \sum_{j=1}^n \alpha_j^{(1/\lambda) t} \sum_{k\in \ZZ} e^{\frac{2\pi i}{\lambda}kt} \\
&=\sum_{j=1}^n \alpha_j^{(1/\lambda) t} \lambda \sum_{m\in \ZZ} \delta_{m\lambda}\\
&=\lambda \sum_{m\in \ZZ} \left ( \sum_{j=1}^n \alpha_j^{m} \right )\, \delta_{m\lambda} \\
&= \lambda \sum_{m\in \ZZ} S_m \, \delta_{m\lambda}
\end{align*}

Now, using the Poisson-Newton formula in $\RR_+^*$
$$
 \sum_\rho e^{\rho t}= \sum_{\bk \in \Lambda} \langle \boldsymbol{\lambda} , \bk
\rangle \, b_{\bk} \ \delta_{\langle \boldsymbol{\lambda} ,\bk\rangle } \, ,
 $$
taking into account the repetitions in the right side, and that
$\boldsymbol{\lambda}=(\lambda_1,\ldots, \lambda_n)= (\lambda ,2\lambda ,\ldots, n\lambda )$,
we have using the formula (\ref{eqn:bs}) for the $b_\bk$
$$
S_m=m \sum_{ k_1+2k_2+\ldots +nk_n=m} b_\bk = m \sum_{ k_1+2k_2+\ldots +nk_n=m} \frac{(||k||-1)!}{\prod_j k_j} \prod_j \Sigma_j^{k_j} \, ,
$$
which gives the explicit Newton relations. Moreover, Newton relations are equivalent to the Poisson-Newton formula in $\RR_+^*$ in this case. 

For example, for $m=4$,
$$
S_4=4\, (b_{(4,0,0,0)} +b_{(2,1,0,0)} + b_{(1,0,1,0)} + b_{(0,2,0,0)} + b_{(0,0,0,1)} ) \ ,
$$
and from
\begin{align*}
 b_{(4,0,0,0)} &= \frac14 \Sigma_1^4 \\
 b_{(2,1,0,0)} &=  - \Sigma_1^2 \Sigma_2 \\
 b_{(1,0,1,0)} &= \Sigma_1 \Sigma_3 \\
 b_{(0,2,0,0)} &= \frac12 \Sigma_2^2 \\
 b_{(0,0,0,1)} &= - \Sigma_4
\end{align*}
we get 
$$
S_4= \Sigma_1^4 -4\Sigma_2\Sigma_1^2 +4 \Sigma_3 \Sigma_1 +2\Sigma_2^2 -4\Sigma_4 \ .
$$

\subsection{Abel-Plana summation formula.}

The full Poisson-Newton formula for $f(s)=1-e^{-\lambda s}$
is a ``half'' classical Poisson formula and can be
written for $\lambda=1$ as
\begin{equation}\label{eqn:half_poisson}
\sum_{n\geq 0} \delta_n  = \frac12 \delta_0 + \left (\sum_{k\in \ZZ} e^{2\pi i kt}\right ) {\unit}_{\RR_+} \  .
\end{equation}

We check now that this is
essentially the Abel-Plana summation formula (see \cite{AB}, \cite{P} and \cite{BFSS}). The Abel-Plana
summation formula compares an infinite sum with the corresponding integral.

\begin{theorem}\textbf{(Abel-Plana summation formula)}
 Let $f$ be a holomorphic function in a domain containing the right half plane $\HH_+=\{\Re z >0\}$ and
continuous in $\overline{\HH}_+$ with
$$
\lim_{y\to +\infty} |f(x\pm i y)| e^{-2\pi y} =0
$$
uniformly on compact sets of $x$, and such that
$$
\int_0^{+\infty} |f(x+iy)-f(x-iy)| e^{-2\pi y} dy
$$
exists for $x\in  \RR_+$ and tends to $0$ when $x\to +\infty$.

We have
$$
\lim_{N\to +\infty} \left (\sum_{n=0}^{N} f(n) -\int_0^{N+1/2} f(t) \, dt \right ) =
\frac12 f(0) +i \int_0^{+\infty}
\frac{f(it)-f(-it)}{e^{2\pi t}-1} \, dt \ ,
$$
and when the sum converges,
$$
\sum_{n=0}^{+\infty} f(n) -\int_0^{+\infty } f(t) \, dt = \frac12 f(0) +i \int_0^{+\infty}
\frac{f(it)-f(-it)}{e^{2\pi t}-1} \, dt \ .
$$
\end{theorem}

\begin{proof}
We apply the full Poisson-Newton formula (\ref{eqn:half_poisson}) to a smooth truncation $f_N$ of $f$ with compact support, so in particular $f_N$ is in the Schwarz class. 
We consider $f_N (x+iy)=\chi_N (x) f(x+iy)$ with $\chi_N(x)$ equal to $1$ in a neighborhood
of $[0,N]$ and vanishing outside $[-\epsilon, N+1/2+\epsilon]$ with $0<\epsilon <1/2$. We also require
that $||f \unit_{[0,N]}-f_N||_{L^1(\RR)}\to 0$ when $N\to +\infty$. The function $f_N$ coincides with $f$
and is holomorphic in a neighborhood of the strip $0\leq \Re z \leq N+1/2$. Applying to $f_N$ the half
Poisson formula, we have
\begin{align*}
\sum_{n=0}^{+\infty} f_N(n) &= \sum_{n=0}^{N} f (n) \\
&= \frac12 f(0) + \sum_{k\in \ZZ} \int_0^{+\infty} f_N(t)e^{2\pi i kt} dt  \\
&= \frac12 f(0) + \int_0^{+\infty} f_N(t) \, dt + \sum_{k=1}^{+\infty} \int_0^{+\infty} f_N(t)e^{2\pi i kt} dt
 - \sum_{k=1}^{+\infty} \int_0^{-\infty} f_N(-t)e^{2\pi i kt} dt \\
&= \frac12 f(0) + \int_0^{N+1/2} f(t) \, dt +  \sum_{k=1}^{+\infty} \left (\int_0^{N}
f(t)e^{2\pi i kt} dt
 - \int_0^{-N} f(-t)e^{2\pi ikt} dt \right ) + o(1) \ ,
\end{align*}
where $o(1)\to 0$ when $N\to +\infty$.

For $R>0$, in the domain of holomorphy of $f$, and using Cauchy theorem, 
we decompose each integral into three line integrals over a rectangular contour:

\begin{align*}
 \int_0^{N} f(t)e^{2\pi i kt} dt &= \int_0^{iR} f(z)e^{2\pi i kz} dz +\int_{iR}^{N+iR}
f(z)e^{2\pi i kz} dz +\int_{N+iR}^{N} f(z)e^{2\pi i kz} dz \\
 \int_0^{-N} f(-t)e^{2\pi i kt} dt &= \int_0^{iR} f(-z)e^{2\pi i kz} dz +\int_{iR}^{-N+iR}
f(-z)e^{2\pi i kz} dz +\int_{-N+iR}^{-N} f(-z)e^{2\pi i kz} dz
\end{align*}
When $R\to +\infty$ the second integral of each line tends to $0$ because of the first hypothesis on $f$.
The substraction of the third integrals gives
\begin{align*}
 \int_{N+iR}^{N} f(z)e^{2\pi i kz} dz - \int_{-N+iR}^{-N} f(-z)e^{2\pi i kz} dz  &=
i\int_R^0 f(N+iu) e^{-2\pi ku} du -i\int_R^0 f(N-iu) e^{-2\pi ku} du\\
&=-i\int_0^R (f(N+iu)-f(N-iu)) e^{-2\pi ku} du \ .
\end{align*}
And the second hypothesis on $f$ shows that this last expression tends to $0$ when $N\to +\infty$. So in the limit we are left with

$$
\sum_{n=0}^{N} f(n) - \int_0^{N+1/2} f(t) \, dt =\frac12 f(0)  + i\int_0^{+\infty}
\frac{f(it)-f(-it)}{e^{2\pi t}-1} \, dt +o(1)\ ,
$$
and the result follows.
\end{proof}

\subsection{Euler-MacLaurin formula and generalizations.}

We have from equation (\ref{eqn:variation})

\begin{proposition}\label{prop:5.2}
For $m\geq d$ we have
$$
\frac{D^m}{Dt^m} K_{m,A}(t, \sigma ) =K_{0,A}(t,\sigma)+\sum_{j=1}^m K_{j,A}(0,\sigma) \delta_0^{(j-1)} \ .
$$
And passing to the limit in $A$ we have
$$
e^{-\sigma t} W(f)=-\sum_{j=d+1}^m K_{j}(0,\sigma) \delta_0^{(j-1)} +\frac{D^m}{Dt^m} K_{m}(t, \sigma ) \ .
$$
\end{proposition}

The half Poisson formula can be written as
$$
\sum_{n=0}^{+\infty} \delta_n = \frac12 \delta_0 +W(1-e^{-s}) \ .
$$
We prefer to work with $g(s)=\frac{1-e^{-s}}{s}$ in order to make the simplest choice $\sigma=0$.
Observing that
$$
W(1-e^{-s})={\emph \unit}_{\RR_+} + W\left (\frac{1-e^{-s}}{s}\right )  \ ,
$$
we have
$$
\sum_{n=0}^{+\infty} \delta_n = \frac12 \delta_0 +{\emph \unit}_{\RR_+} + W\left (\frac{1-e^{-s}}{s}\right ) \ .
$$
Now we apply the formula of Proposition \ref{prop:5.2} 
with $\sigma =0$ and $d=1$ to $g(s)=\frac{1-e^{-s}}{s}$.
Then we have for $j \geq d+1$
$$
K_j(0,0)=\sum_{k\in \ZZ^*} \frac{1}{(2\pi i k)^j} \ ,
$$
thus $K_j(0,0)=0$ when $j$ is odd. Recalling that
$$
\zeta(2l)=(-1)^{l+1} \frac{B_{2l} (2\pi)^{2l}}{2 (2l)!} \ ,
$$
where the $(B_n)$ are the Bernouilli numbers, 
we have for $j$ even
$$
K_{2l}(0,0) =  2 (-1)^l (2\pi)^{-2l} \sum_{k\geq 1} k^{-2l} =2 (-1)^l (2\pi)^{-2l} \zeta (2l) =-\frac{B_{2l}}{(2l)!} \ .
$$
Thus for $m$ even, or by replacing $m$ by $2m$

\begin{theorem} \textbf{(Infinite distributional Euler-MacLaurin formula)}\label{thm:infinite EM}
We have
$$
\sum_{n=0}^{+\infty} \delta_n = \frac12 \delta_0 +{\unit}_{\RR_+} +\sum_{l=1}^{m}
\frac{B_{2l}}{(2l)!}  \delta_0^{(2l-1)} +\frac{D^{2m}}{Dt^{2m}} K_{2m}(t, 0 )\ .
$$
\end{theorem}

This strongly reminds Euler-MacLaurin formula in distributional form. We apply this to a test function $f$. Observe that
$$
\langle \delta_0^{(2l-1)} , f\rangle = -f^{(2l-1)} (0) \ ,
$$
and 
\begin{align*}
\left \langle \frac{D^{2m}}{Dt^{2m}} K_{2m} , f \right \rangle &= \langle  K_{2m} , f^{(2m)} \rangle \\
&= \left \langle \left (\sum_{k\in \ZZ^*} \frac{e^{2\pi i kt}}{(2\pi i k)^{2m}} \right ) {\unit}_{\RR_+}, f^{(2m)} \right \rangle \\
&=(-1)^m \sum_{k\geq 1} \int_0^{+\infty } \frac{e^{2\pi i kt} + e^{-2\pi i kt}}{(2\pi k)^{2m}} f^{(2m)}(t) \, dt \ .
\end{align*}

\begin{theorem} \textbf{(Infinite Euler-MacLaurin formula)}
For a test function $f$ in the Schwarz class
$$
\sum_{n=0}^{+\infty} f(n) = \frac12 f(0) +\int_0^{+\infty} f(t) dt -\sum_{l=1}^{m}
\frac{B_{2l}}{(2l)!}  f^{(2l-1)}(0) + (-1)^m
\sum_{k\geq 1} \int_0^{+\infty } \frac{e^{2\pi i kt} + e^{-2\pi kt}}{(2\pi k)^{2m}}
f^{(2m)}(t) \, dt\ .
$$
\end{theorem}

In order to recover the standard finite Euler-Maclaurin formula, we substract the infinite formula from itself
after a translation by an integer $N\geq 0$.

\begin{theorem} \textbf{(Distributional Euler-MacLaurin formula)}
We have
$$
\sum_{n=0}^{N} \delta_n = \frac12 (\delta_0 +\delta_N) +{\emph \unit}_{[0,N]} +\sum_{l=1}^{m}
\frac{B_{2l}}{(2l)!}  (\delta_0^{(2l-1)}-\delta_N^{(2l-1)}) +  \frac{D^{2m}}{Dt^{2m}}
\left (\sum_{k\in \ZZ^*} \frac{e^{2\pi i kt}}{(2\pi  k)^{2m}} \right ) {\unit}_{[0,N]} \ .
$$
\end{theorem}

A particular case is a finite version of the ``half'' classical Poisson formula that we used in
the previous section for deriving Abel-Plana formula.

\begin{theorem}\textbf{(Finite half Poisson formula)}
We have
\begin{equation}
\sum_{n= 0}^N \delta_n  = \frac12 (\delta_0 +\delta_N) + {\emph \unit}_{[0,N]} +\left (\sum_{k\in \ZZ^*} e^{2\pi i kt}\right ) {\emph\unit}_{[0,N]} \  .
\end{equation}
\end{theorem}

Now, we check that this finite distributional Euler-MacLaurin formula gives the usual formula.

\begin{theorem}
 For $N, m \geq 0$ and $f$ a $C^{2m}$ function in $[0,N]$ we have
\begin{align*}
\sum_{n=0}^N f(n) &=\int_0^N f(t) dt +\frac12 (f(0)+f(N)) +\sum_{k=1}^m \frac{B_{2k}}{(2k)!} \left (
f^{(2k-1)}(N)-f^{(2k-1)}(0) \right ) \\
&+(-1)^m \sum_{k=1}^{+\infty} \int_0^N \frac{e^{2\pi i kt} + e^{-2\pi i kt}}{(2\pi k)^{2m}} f^{(2m)}(t) dt
\end{align*}
\end{theorem}

\begin{proof}
  By density of the Schwarz class in the $C^{2m}$ topology of $C^{2m}$ functions in $[0,N]$, it is 
enough to prove the formula for a function $f$ in the Schwarz class. Then it is a direct application of the previous finite
distributional formula.
\end{proof}

With our approach we can generalize Euler-MacLaurin formula by choosing the parameter $\sigma\not= 0$.

\begin{theorem}For $\sigma\not= 0$ we have:
 \begin{align*}
\sum_{n=0}^{+\infty } \delta_n &=\frac12 \delta_0 +\unit_{\RR_+} -\sum_{j=d+1}^m e^{-\sigma t} 
K_j(0,\sigma) \delta_0^{(j-1)} + e^{-\sigma t} \frac{D^m}{Dt^m} K_m(t,-\sigma) \\
&=\frac12 \delta_0 +\unit_{\RR_+} -\sum_{k=0}^{m-1} L_k(\sigma) \, \delta_0^{(k)} + e^{-\sigma t} \frac{D^m}{Dt^m} K_m
\end{align*}
where
$$
L_k(\sigma) =\sum_{l=\max (d+1, k+1)}^m \binom{l-1}{k} \sigma^{l-1-k} K_l(0, -\sigma)  \ ,
$$
and
$$
K_l(0,-\sigma)=\sum_{k\in \ZZ^*} \frac{1}{(2\pi i k+\sigma)^l}=(2\pi i)^{-l} 
\left ( \zeta \left (l, \frac{\sigma}{2\pi i} \right )
+(-1)^l \zeta \left (l, -\frac{\sigma}{2\pi i}\right ) \right) \ ,
$$
where we use Hurwitz zeta function
$$
\zeta(s,q)=\sum_{n=0}^{+\infty} \frac{1}{(n+q)^s} \ .
$$
\end{theorem}

\begin{remark}
The approach presented to the classical Euler-MacLaurin summation formula consists in applying our Poisson-Newton formula
to the simplest Dirichlet series $f(s)=1-e^{-s}$. It is clear that this admits an infinite number of generalizations applying it 
to arbitrary Dirichlet series. We will develop these questions elsewhere, but we want to notice that in general we get summation formulas over
the set of frequencies of well behaved Dirichlet series. In particular, we can obtain summation formulas for quasi-peroiodic
sequences, etc
\end{remark}

\subsection{Ramanujan theory of the constant of a series.}

Ramanujan developed an heuristic theory of summation of series from Euler-MacLaurin 
formula. He associated to some divergent 
series its ``constant'' that in his words ``It is like the center of gravity of a body'' 
(see \cite{RA} p. 40 of the first notebook). 
We refer to chapter 6 of \cite{Berndt} 
for an overview of the theory as exposed in Ramanujan's first notebook. As explained by Berndt, 
Hardy attempted to formalize Ramanujan's theory in \cite{Hardy2} in order to make proper sense 
of Ramanujan's constant for summable series. This does not seem quite in the spirit of Ramanujan. Its 
theory makes sense for sums of the type
$$
\sum_{n=1}^{+\infty} f(n) \ ,
$$
where the general term $f(n)$ is given by an analytic function $f$ with global properties. 

Writing in the following form our infinite distributional form of Euler-McLaurin, theorem \ref{thm:infinite EM},
$$
\sum_{n=1}^{+\infty} \delta_n -{\unit}_{\RR_+} = -\frac12 \delta_0  +\sum_{l=1}^{m}
\frac{B_{2l}}{(2l)!}  \delta_0^{(2l-1)} +\frac{D^{2m}}{Dt^{2m}} K_{2m}(t, 0 )\ .
$$
we see that it is natural to define the Ramanujan class of tests functions $f$ with global properties for which
$$
\lim_{m to + \infty}\left \langle\frac{D^{2m}}{Dt^{2m}} K_{2m}(t, 0 ), f \right \rangle =0 \ .
$$
For these it is natural to define the Ramanujan distribution $RC$ of infinite order and the Ramanujan 
constant of the series as Ramanujan does by
$$
RC(f)=\langle RC , f\rangle= -\frac12 f(0)  -\sum_{l=1}^{+\infty}
\frac{B_{2l}}{(2l)!}  f^{(2l-1)}(0) \ . 
$$
In some sense we can write
$$
RC(f)=\sum_{n=1}^{+\infty} f(n)- \int_0^{+\infty} f(t) \ dt \ ,
$$
and $RC(f)$ weights the ``equilibrium'' between the divergence of the series and the divergence of the integral.

In this way we can develop rigorously Ramanujan theory, and this will be done elsewhere.

\subsection{Location of the divisor of a Dirichlet series.}

From the Poisson formula, we get that the distribution $W(f)$ is an atomic distribution on $\RR_+^*$.
Thus the sum of exponentials associated to the zeros cannot be a convergent series for $t\in \RR_+^*$. But
the following lemma gives a simple condition which would imply the analytic convergence of the sum.
For $\theta_1 < \theta_2 $, $\theta_2-\theta_1 < \pi$, denote by $C(\theta_1, \theta_2)$ the cone
of values of $s\in \CC$ with
$$
\theta_1 < \Arg \ s < \theta_2 \, .
$$

\begin{lemma} \label{lemma:cone}
If $\{\rho\} \subset C(\theta_1, \theta_2)$, then
$$
W(t)=\sum_\rho n_\rho e^{\rho t} \ ,
$$
is a holomorphic function in $C(\pi/2-\theta_1 , 3\pi/2 -\theta_2)$.
\end{lemma}

\begin{proof}
For $t\in C(\pi/2-\theta_1 , 3\pi/2 -\theta_2)$ we have $\Re (\rho \, t) <0$, whence
$$
\left |e^{\rho t} \right | <1 \ ,
$$
and the series $K_\ell(t)$ defined in (\ref{eqn:Kl}) is holomorphic in that region, so the results follows.
\end{proof}

From this we obtain the following straightforward corollary:

\begin{corollary} \label{cor:divisor}
 The divisor of any Dirichlet series cannot be contained in a cone $C(\theta_1, \theta_2)$ for
$\pi/2 <\theta_1 < \theta_2 <3\pi/2 $.
\end{corollary}

In particular, once established the functional equation for the Riemann zeta function and its unique
pole at $s=1$, we get that it must  have an infinite number of non-real zeros in the critical strip
$0<\Re s < 1$ (this is not a hard result in any case).

The above motivates the definition of $\theta$-distribution:

\begin{definition}
 A $\theta$-distribution is a Newton-Cramer distribution $W$ in $\RR$,  associated to a finite order
divisor $\sum_\rho n_\rho \, \rho$ 
such that the series
$$
W(t)=\sum_\rho n_\rho e^{\rho t}
$$
is absolutely convergent and defines an analytic function in compact sets of $\RR_+^*$.
\end{definition}

We also talk about germs of $\theta$-distributions for the germ to the right of $0$ of $\theta$-distributions.
Also we define left-directed divisors:

\begin{definition}
 A divisor $D=\sum_\rho n_\rho \, \rho$ is left-directed if it is contained in a left cone
$C(\theta_1, \theta_2)$ with $\pi/2 <\theta_1 < \theta_2 <3\pi/2 $.
\end{definition}

Thus we proved above that no non-constant Dirichlet series has a left-directed divisor and:

\begin{proposition}
 If $f$ is a meromorphic function of finite order with a left-directed divisor, then its Newton-Cramer
distribution $W(f)$ is a $\theta$-distribution.
\end{proposition}

\subsection{Hadamard regularization of $\theta$-distributions.}

For a $\theta$-distribution, one may ask naturally what is the relation of the Hadamard regularization
at $0$ of the germ of analytic function $W(t)$ on $]0,\epsilon[$ and the distribution $W$.
We show that
by the local Hadamard regularization we recover the local structure of the germ defined by $W$ at $0$.
But the theorem we present below has also a global meaning in $\RR_+$.
It shows that the Hadamard formula gives a global integral formula for
the full distribution $W$.

\begin{theorem}
Let $W=W(\sigma_1 , d)$ be a $\theta$-distribution in $\RR$ associated
to a left-directed divisor with convergence exponent $d\geq 2$.
If $\sigma_1 \in \RR$ is the vertex of a left cone containing the divisor, then
there are constants $C>0$ and $\epsilon >0$ such that for $t>0$
$$
|W(t)|\leq C t^{-d} e^{(\sigma_1-\epsilon) t} \, .
$$
The distribution $W$ can be computed by Hadamard regularization at $0$
for a test function $\varphi$ such that $\psi=e^{\sigma_1 t} \varphi\in \cS$ is
in the Schwartz class,
\begin{align*}
\langle W , \varphi \rangle &= \int_0^{+\infty} W(t)
 e^{-\sigma_1 t}
 \left(\psi (t)-\psi(0) -\psi'(0) t -\ldots -\frac{\psi^{(d-2)}(0)}{(d-2)!} t^{d-2} \right ) \, dt \\
&= \int_0^{+\infty} W(t)
 \left(\varphi (t)-e^{-\sigma_1 t} \sum_{l=0}^{d-2} \frac{1}{l!} \left (\sum_{k=0}^{d-l-2}
\frac{(\sigma_1 t)^k}{k!}\right ) \varphi^{(l)}(0) \right ) \, dt \\
&= \int_0^{+\infty} W(t)
 \left(\varphi (t)-\sum_{l=0}^{d-2} \frac{1}{l!} \varphi^{(l)}(0) t^l \right ) \, dt
+\sum_{l=0}^{d-2 } c_l(\sigma_1) \, \varphi^{(l)}(0) \ ,
\end{align*}
where
$$
c_l(\sigma_1)=\frac{1}{l!} \int_0^{+\infty} W(t) t^l e^{-\sigma_1 t} R_{d-1-l}(\sigma_1  t) \, dt \ ,
$$
and
$$
R_n(x)= e^{x} -\sum_{k=0}^n \frac{x^k}{k!}=\sum_{k= n+1}^{+\infty }  \frac{x^k}{k!} \ .
$$
Therefore the distribution $W$ only differs by a contribution at $0$ from the classical Hadamard regularization $H$,
$$
W= H +\sum_{l=0}^{d-2 } c_l(\sigma_1) \, \delta_0^{(l)} \ .
$$
\end{theorem}

Note that we can extend the integrals to the whole of $\RR$ since $W=0$ in $\RR_-^*$.

\begin{proof}
 From definition \ref{def:functional} the divisor is contained in a left strict cone, so by
proposition \ref{lemma:cone} we have that $W$ is holomorphic in a $t$-cone containing the positive real axis,
in particular it is analytic in $\RR^*_+$. Moreover $e^{\sigma_1 t} W$ is the $d$-th derivative of
the uniformly bounded holomorphic function $K_d$ in that cone. Since $\sigma_1$ is not part of the divisor, we
can consider $\sigma_1-\epsilon$ instead of $\sigma_1$ and open slightly the cone in order to have a left
cone with vertex at $\sigma_1-\epsilon$ still containing the divisor. Thus we do have the better estimate
$$
|K_d(t)|\leq C e^{-\epsilon \Re t} \ ,
$$
because we have in the new cone in the $t$-plane
$$
\Re ((\rho-\sigma_1) t)\leq -\epsilon \Re t <0 \ .
$$

For $t_0 \in \RR_+$, we have that the distance $d(t_0)$ to the boundary of the $t$-cone is
$\geq C |t_0|$, and the estimate of the theorem results from Cauchy estimate,
$$
|W(t_0)| =\frac{1}{d!} e^{\sigma_1 t_0} \left | \int_{C(t_0 , d(t_0))} \frac{K_d(z)}{(z-t_0)^{d+1}} \ 
\frac{dz}{2\pi i}\right |
\leq C |t_0|^{-d} e^{(\sigma_1 -\epsilon) t_0}  \ .
$$
Note that in the same way we obtain the control of the derivatives, for $l=0,1,\ldots , d$ and $t>0$
$$
\left | \frac{d^l}{dt^l} (K_d(t)-K_d(0)) \right | 
=\left | \frac{d^{l+1}}{dt^{l+1}} K_{d-1}(t) \right | 
\leq C t^{-l-1} e^{-\epsilon t}   \ .
$$ 
Observing that $K_d(t)-K_d(0) \to 0$ for $t\to 0$ in the cone, we obtain that for $t\to 0$, and $l=1,2,\ldots
, d-1$
$$
\frac{d^l}{dt^l} (K_d(t)-K_d(0)) =o(t^{-l})  \ .
$$

Define 
$$
\tilde \psi(t)= \psi (t)-\psi(0) -\psi'(0) t -\ldots -\frac{\psi^{(d-2)}(0)}{(d-2)!} t^{d-2} \ ,
$$
and observe that near $0$
$$
\frac{d^l}{dt^l}\tilde \psi (t) =\cO (t^{d-1-l}) \ .
$$
Therefore, for $l=1,2,\ldots, d-1$ we have 
$$
\lim_{t\to 0} \frac{d^l}{dt^l} (K_d(t)-K_d(0))\frac{d^{d-l-1}}{dt^{d-l-1}}\tilde \psi (t) =0 \ .
$$

Using that for $t>0$
$$
W(t)=e^{\sigma_1 t} \frac{d^d}{dt^d} (K_d(t)-K_d(0)) ,
$$
we have performing $d$ integrations by parts
 \begin{align*}
 \langle W , \varphi \rangle
  & =\langle K_{d}(t)-K_d(0), (-1)^d \frac{d^d}{dt^d}(e^{\sigma_1 t} \varphi) \rangle \\
  &= \int_0^{+\infty}
  (-1)^d (K_{d}(t)-K_d(0)) \frac{d^d}{dt^d}(\psi) \, dt \\
  &=  \int_0^{+\infty}
  (-1)^d (K_{d}(t)-K_d(0)) \frac{d^d}{dt^d}
  \left(\psi (t)-\psi(0) -\psi'(0) t -\ldots -\frac{\psi^{(d-2)}(0)}{(d-2)!} t^{d-2} \right ) \, dt \\
  &=  \int_0^{+\infty}
  \left(\frac{d^d}{dt^d} K_{d}(t)\right)
  \left(\psi (t)-\psi(0) -\psi'(0) t -\ldots -\frac{\psi^{(d-2)}(0)}{(d-2)!} t^{d-2} \right ) \, dt
 \end{align*}
where we used the estimates for the vanishing at the evaluations at $0$ and 
$+\infty$.
This proves the Hadamard regularization formula stated in the theorem.
\end{proof}

\section{Functional equations.} \label{sec:functional}

We give in this section a precise definition of the property of ``having a functional equation''. We know no reference
in the classical literature.

We say that the divisor $D_1$ is contained in the divisor $D_2$, and denote this by
$$
D_1 \subset D_2 \ ,
$$
if any zero, resp.\ pole,  of $D_1$ is a zero,
resp.\ pole, of $D_2$, and $|n_\rho (D_1)| \leq |n_\rho (D_2)|$ for all $\rho \in \CC$.

\begin{definition} \label{def:functional}
 The meromorphic function $f$ has a functional equation if there exists $\sigma^*\in \RR$ and a
divisor $D\subset \Div f$ contained in a left cone
$\sigma_1+C(\theta_1, \theta_2)$, with $\pi/2 <\theta_1 < \theta_2 <3\pi/2 $, 
such that $\Div f -D$ is infinite and symmetric with respect to the vertical line $\{\Re s =\sigma^* \}$.
\end{definition}

\begin{proposition}
 If $f$ has a functional equation and 
the divisor of $f$ is contained in a left half plane then $\sigma^* \in \RR$ is unique.
\end{proposition}

\begin{proof}
Otherwise, if they were two distinct values $\sigma^*$, then $\Div f$ would have an infinite 
subdivisor invariant by a real 
translation and this contradicts the hypothesis that the divisor of $f$ is contained in a left half plane.
\end{proof}

\begin{proposition} \label{prop:6.4}
 If $f$ has a functional equation and the divisor of $f$ is contained in a 
left half plane then $\Div f -D$ is contained 
in a vertical strip.  The minimal strip $\{ \sigma_- <\Re s < \sigma_+ \}$,
$\sigma_+ < \sigma_1$, with this property is the critical strip and
$\sigma^*=\frac{\sigma_-+\sigma_+}{2}$ is its center.
\end{proposition}

\begin{proof}
Since $\Div (f)$ has no zeros nor poles for $\Re s >\sigma_1$, the divisor of $\Div f -D$ is contained
in a vertical strip due to the symmetry. The minimal vertical strip has to be compatible with the functional equation, 
hence $\sigma^*=\frac{\sigma_-+\sigma_+}{2}$.
\end{proof}

\begin{proposition}
 If $f$ has a functional equation and the divisor of $f$ is contained in a left half plane then there is a unique 
minimal divisor $D$ (i.e., with $ |n_\rho (D) |$ minimal for all $\rho \in \CC$), and a unique decomposition 
$D=D_0+D_1$, $D_0$ and $D_1$ with disjoint supports, with $D_0$ contained in a left cone
$\sigma^*+C(\theta_1, \theta_2)$, with $\pi/2 <\theta_1 < \theta_2 <3\pi/2 $, and $D_1$ a finite divisor 
contained in the half plane $\{\Re s >\sigma^* \}$, such that $\Div f -D$ is infinite and symmetric with 
respect to the vertical line $\{\Re s =\sigma^* \}$.
\end{proposition}

\begin{proof}
We start with $D$ minimal as in the definition, and we define $D_0$ to be the part of $D$ to the left of  
$\{\Re s =\sigma^* \}$ and $D_1$ the remaining part. It is easy to see that $D_0$ is contained in a left cone with 
vertex at $\sigma^*$.
\end{proof}

\begin{theorem}
 If $f$ has a functional equation and the divisor of $f$ is contained in a left half plane then there exists 
a meromorphic function $\chi$ with $\Div \chi =D \subset \Div f$, such that the function $g(s)=\chi(s) f(s)$ satisfies
the functional equation 
 $$
g(2\sigma^* - s)=g(s) \ .
 $$
Moreover, we can write $\chi =\chi_0 \cdot R$ with $\Div \chi_0 =D_0-\tau^*D_1$ and $\Div R =D_1 + \tau^*D_1$,
where $\tau$ is the reflexion along $\Re s=\sigma^*$,  
and $R$ is a unique rational function up to a non-zero multiplicative constant.

The meromorphic function $\chi$ (or $\chi_0$) is uniquely determined
up to a factor $\exp h(s-\sigma^*)$ where $h$ is an even entire function. 
If $f$ has convergence exponent $d <+\infty$,  
then  we can take $\chi$  of  convergence exponent $d$, and then $\chi$ 
is uniquely determined up to a factor $\exp P(s-\sigma^*)$
where $P$ is an even polynomial of degree
less than $d$. In particular, when $f$ is of order $1$ then $\chi$ 
and $\chi_0$ are uniquely determined up to a
non zero multiplicative constant.
\end{theorem}

\begin{proof}
 From Proposition \ref{prop:6.4} we know that $\sigma^*$ is uniquely determined as the
center of the critical strip
(which is defined only in terms of the divisor of $f$).
Translating everything by $\sigma^*$ we can assume
that $\sigma^*=0$. By minimality the divisor of $\chi$ is
uniquely determined. Then $\chi$ is uniquely determined up to a factor
$\exp h(s)$ where $h$ is an entire
function. If $\hat \chi (s)=(\exp h(s) )\chi (s)$ gives also a functional equation for $f$, then we have
$$
f(s)=\frac{\chi (-s)}{\chi (s)} f(-s)= \frac{\chi (-s)}{\chi (s)} \frac{\hat \chi (s)}{\hat \chi (-s)} f(s).
$$
Therefore
$$
\exp(h(s)-h(-s)) =1,
$$
so for some $k\in \ZZ$,
$$
h(s)-h(-s)=2\pi i k \, .
$$
Specializing for $s=0$ we get $k=0$ and $h$ is even.

When $f$ is of  convergence exponent $d<+\infty$, and since the divisor 
of $\chi$ is contained in the divisor of $f$, then we can take $\chi$ of convergence exponent at most $d$. 
\end{proof}

If $f$ is real analytic, then it is easy to see that $\chi$ must be real analytic up to the Weierstrass
factor. We will always choose $\chi$ to be real analytic.
Then $g=\chi f$ is real analytic, and we have a four-fold symmetry and $g$ is symmetric with respect to the
vertical line $\Re s =\sigma^*$.

It is of interest to dissociate the contribution to the distribution $W(f)$ of the part of the divisor
of $f$ that comes from the divisor of $\chi_0$. This part is the $\theta$-distribution $W(\chi_0)$ for which we
have Hadamard regularization formula. For a test function $\varphi$ such that
$\psi=e^{\sigma_1 t} \varphi\in \cS$ is
in the Schwartz class,

\begin{align*}
\langle W(\chi_0) , \varphi \rangle &= \int_0^{+\infty} W(\chi_0)(t)
 e^{-\sigma_1 t}
 \left(\psi (t)-\psi(0) -\psi'(0) t -\ldots -\frac{\psi^{(d-2)}(0)}{(d-2)!} t^{d-2} \right ) \, dt \\
&=\int_0^{+\infty} W(\chi_0)(t)
 \left(\varphi (t)-\sum_{l=0}^{d-2} \frac{1}{l!} \varphi^{(l)}(0) t^l \right ) \, dt
+\sum_{l=0}^{d-2} c_l(\sigma_1) \, \varphi^{(l)}(0) \ ,
\end{align*}
where
$$
c_l(\sigma_1)=\frac{1}{l!} \int_0^{+\infty} W(\chi_0)(t) t^l e^{-\sigma_1 t} R_{d-1-l}(\sigma_1  t) \, dt \ ,
$$
where
$$
R_n(x)= e^{x} -\sum_{k=0}^n \frac{x^k}{k!}=\sum_{k= n+1}^{+\infty }  \frac{x^k}{k!} \ .
$$

\subsection*{Example.}

For the Riemann zeta function $f(s)=\zeta(s)$ we have $\sigma^*=1/2$, $\sigma_-=0$, $\sigma_+=1$, 
$D_0=-2\NN^*$, $D_1= \{1 \}$, and
\begin{align*}
\chi(s) &=\pi^{-s/2}\Gamma (s/2) s(s-1), \\ 
\chi_0(s) &=\pi^{-s/2}\Gamma(s/2), \\  
P(s) &=s(s-1). 
\end{align*}
Note that 
$$
g(s)=\chi(s)\zeta(s)=\pi^{-s/2}\Gamma (s/2) s(s-1) \zeta(s)=2 \xi (s) \ .
$$
(using Riemann's classical notation for $\xi$).

\bigskip

Next, we determine when a finite Dirichlet series satisfies a functional equation.

\begin{proposition}  \label{prop:functional-eqn}
A finite Dirichlet series
$$
f(s)=1+\sum_{n=1}^N a_n e^{-\lambda_n s} \ ,
$$
satisfies a functional equation if and only if it is
of the form
 $$
 f(s)= e^{\mu s} \sum_{i=0}^{[(N-1)/2]}  a_i (e^{(-\lambda_i+\mu) s} +
  c\, e^{(\lambda_i-\mu) s})\, ,
 $$
where $c=1$ if $N$ is even, $c=\pm 1$ if $N$ is odd.
\end{proposition}

\begin{proof}
 The Dirichlet series $f(s)$ is of order $1$. Suppose that there is
 some $\chi(s)$ of order $1$ with zeros and poles in a left  cone
 such that $g(s)=\chi(s) f(s)$ is symmetric
with respect to $\Re s=\sigma^*$. By translating, we can assume $\sigma^*=0$.

The zeros of $f(s)$ lie in a strip, since $e^{-\lambda_n s} f(-s)$ is also a Dirichlet series.
Therefore $\chi(s)$ has finitely
many zeros and poles, and hence $\chi(s)=\frac{Q_1(s)}{Q_2(s)} e^{\mu s}$, for some
polynomials $Q_1(s), Q_2(s)$. The functional equation $g(s)=g(-s)$ reads
 $$
 Q_1(s)Q_2(-s) \sum_{n=0}^N a_n e^{(\mu-\lambda_n) s} =
 Q_2(s)Q_1(-s) \sum_{n=0}^N a_n e^{(\lambda_n-\mu) s} \, ,
 $$
where we have set $a_0=1$, $\lambda_0=0$.

{}From this it follows that $Q_1(s)Q_2(-s) =c\, Q_2(s)Q_1(-s)$, $c\in \CC^*$. It
follows easily that $c=\pm 1$.
Also $0,\lambda_1,\ldots, \lambda_N$ is a sequence symmetric with respect to $\mu=\lambda_N/2$.
So $\lambda_{N-i}=2\mu -\lambda_{i}$ and $a_{N-i}=a_i$.

If $N$ even, then $\lambda_{N/2}=\mu$, $c=1$,  and 
 $$
  \sum_{n=0}^N a_n e^{-\lambda_n s} = e^{\mu s} \sum_{i=0}^{N/2-1}  a_i (e^{(-\lambda_i+\mu) s} +
  e^{(\lambda_i-\mu) s}) + a_{N/2} e^{\mu s}\, .
 $$
If $N$ is odd, then 
 $$
  \sum_{n=0}^N a_n e^{-\lambda_n s} = e^{\mu s} \sum_{i=0}^{(N-1)/2}  a_i (e^{(-\lambda_i+\mu) s} +
  c \, e^{(\lambda_i-\mu) s}) \, ,
 $$
where if $c=-1$, we have $\chi(s)=s\, e^{\mu s}$.
\end{proof}

\subsection*{An example without functional equation.}

Consider the elementary Dirichlet series
 \begin{equation} \label{eqn:example}
f(s)= 1 + a_1 e^{-\lambda_1 s}+a_2 e^{-\lambda_2 s}
 \end{equation}
with $0<\lambda_1<\lambda_2$ and $a_j\not=0$. It is an entire function on $\CC$ of order $1$.

If $\lambda_1,\lambda_2$ are rationally dependent, then we may write
$f(s)= 1+ a_1 \left (e^{\lambda s}\right )^{k_1} +a_2 \left (e^{\lambda s}\right )^{k_2}$, for
$\lambda_1=k_1\lambda$, $\lambda_2=k_2\lambda$, $k_1,k_2>0$ and coprime. We can compute
the zeros solving the algebraic equation $1+a_1 X^{k_1}+a_2 X^{k_2} =0$.
Therefore, the zeros of $f(s)$ lie in at most $k_2$ vertical lines, and they
form $k_2$ arithmetic sequences of the same purely imaginary step.

If $\lambda_1,\lambda_2$ are rationally independent, then we cannot compute
the zeros in general. We know that they lie in a half-plane
$\Re s<\sigma_1$. Also $a_2^{-1} e^{\lambda_2 s} f(s)$ converges
to $1$ for $\Re s\to -\infty$. So the zeros of $f(s)$ are located
in a half-plane $\Re s>\sigma_2$. Hence in a strip. By Corollary
\ref{cor:divisor}, there are infinitely many zeros in that strip.

Now, let $(\rho)$ be the set of zeros. Then $\Lambda=\{\bk=(k_{1},k_{2}) \in \NN^2 \, |
\, (k_{1},k_{2}) \neq (0,0)\}$, and
 $$
 b_\bk=\frac{(-1)^{k_1+k_2}}{k_1+k_2} \binom{k_1+k_2}{k_1} a_1^{k_1}a_2^{k_2}
 $$
and
$$
 \sum n_\rho e^{\rho t} = \sum_{\bk}
 (\lambda_1k_1+\lambda_2k_2) b_{\bk} \delta_{\lambda_1k_1+\lambda_2k_2} \,,
$$
on $\RR^*_+$.

By Proposition \ref{prop:functional-eqn},
the Dirichlet series (\ref{eqn:example}) does not have a functional equation
unless $\lambda_2=2\lambda_1$.

\subsection{Functional equation and Hadamard factorization.}

\medskip

Let $f$ be a meromorphic function of finite order which has its divisor contained in a left 
half plane and which has a functional equation. In order to simplify
we assume that $\sigma^*$ is not part of the divisor. For $g(s)=\chi (s) f(s)$ we have
$$
g(2\sigma^*-s)=g(s) \ ,
$$
and when we express this symmetry in the Hadamard factorization, we get
$$
Q_{g,\sigma^*} (2\sigma^*-s)=Q_{g,\sigma^*} (s) \ ,
$$
hence if we write
$$
Q_{g,\sigma^*} (s)=\sum_k a_k (s-\sigma^*)^k \ ,
$$
the symmetry implies that all odd coefficients are zero $a_1=a_3=\ldots =0$.

We observe also that 
$$
Q_g=Q_\chi +Q_f \ ,
$$
and for the discrepancy polynomials
$$
P_g=P_\chi +P_f \ .
$$

Also if the exponent of convergence is $d=2$, pairing the Weierstrass factors for symmetric zeros $\rho$ and $2\sigma^*-\rho$, gives
$$
E_{d-1} \left(\frac{s-\sigma^*}{\rho-\sigma^*}\right ) \cdot 
E_{d-1} \left(\frac{s-\sigma^*}{(2\sigma^*-\rho)-\sigma^*}\right )
=-\frac{1}{\rho-\sigma^*} (s-\rho)(s-(2\sigma^*-\rho)) \ .
$$
Therefore for $d=2$, the discrepancy polynomial is constant and we have
$$
c_0(\chi)= P_\chi =P_{\chi_0}=c_0(\chi_0) \ .
$$
Note that this also holds if we know that $g=2$, in particular when $d=3$ and $f$ is a Dirichlet series by Corollary \ref{cor:genus}.
On the other hand the functional equation implies that 
$$
c_0(g)=0 \ ,
$$
therefore we obtain:

\begin{proposition}\label{functional_eq_0structure}
 Let $f$ be a meromorphic function of exponent of convergence $d=2$, or $g=2$, which has its divisor contained 
in a left half plane and has a functional equation. We assume that $\sigma^*$ is not part of the divisor.

We have
$$
c_0(\chi_0, \sigma^*)+c_0(f, \sigma^*)=0 \ .
$$
\end{proposition}

\subsection{Gauss formula for the logarithmic derivative of the $\Gamma$-function.}

{}From our general Poisson-Newton formula we get in one stroke Gauss and related formulas for the
logarithmic derivative of the $\Gamma$-function.

The $\Gamma$-function was defined by Euler by the limit
$$
\Gamma (s)= \lim_{n\to +\infty} \frac{(n-1)!}{s(s+1)\ldots (s+n-1)} n^s \ ,
$$
which shows that $f(s)=1/\Gamma(s)$ is an entire function with simple zeros at $\rho_n=-n$ for
$n=0,1,2,\ldots$ Hence the exponent of convergence is $d=2$. The divisor of $f$ is left-directed 
and the associated Newton-Cramer
distribution is the $\theta$-distribution
$$
W(1/\Gamma)(t)=\sum_{n=0}^{+\infty} e^{(-n) t} =\frac{1}{1-e^{-t}} \ .
$$

From Euler's definition we can derive directly the Hadamard factorization
$$
f(s)=\frac{1}{\Gamma (s)}= s e^{\gamma s} \prod_{n=1}^{+\infty } \left (1+\frac{s}{n}\right ) e^{-s/n} \ ,
$$
where $\gamma$ is Euler's constant
$$
\gamma=\lim_{n\to +\infty} \sum_{k=1}^n \frac{1}{k} -\log n =\int_0^{+\infty } \left (\frac{1}{1-e^{-t}} -
\frac{1}{t} \right ) e^{-t} \, dt
$$
(see \cite{WW}, p.246 for the integral expression).
This gives the Hadamard interpolation for $\sigma_1=0$. Unfortunately $\sigma_1=0$ is a zero of
$f$ and the Hadamard regularization is divergent at $+\infty$. So we choose $\sigma_1=1$ (for example)
and derive the associated Hadamard factorization
$$
f(s)=\frac{1}{\Gamma (s)}= \frac{1}{s-1} \frac{1}{\Gamma (s-1)}=e^{\gamma (s-1)}
\prod_{n=0}^{+\infty } \left (1-\frac{s-1}{-n-1}\right ) e^{\frac{s-1}{-n-1}} \ .
$$
Thus we have for some $n\in \ZZ$,
$$
Q(s)= 2\pi i n -\gamma + \gamma s \ ,
$$
and the discrepancy is a constant polynomial
$$
P_f(s)=-Q'(s)=-\gamma =c_0 \ .
$$
Applying the general Poisson-Newton formula to $f$ we have
$$
\cL (W(f))=\cL \left ( \frac{e^{- t}}{1-e^{-t}} \right )=-\gamma -\frac{\Gamma'(s)}{\Gamma (s)} \ .
$$
And by the Hadamard regularization formula applied with $\sigma_1=1$, $d=2$ and the test function
$\varphi (t)=e^{-st}$, $\psi (t)=e^{(1-s)t}$,
\begin{align*}
\cL (W(f)) &=\langle W(f), e^{-st}\rangle =\int_0^{+\infty} \frac{1}{1-e^{-t}} e^{-t} \left (e^t
e^{-st} -1 \right ) \, dt \\
&=\int_0^{+\infty} \frac{1}{1-e^{-t}} \left ( e^{-st} -e^{-t} \right ) \, dt
\end{align*}
Therefore we get
$$
\frac{\Gamma'(s)}{\Gamma (s)}=-\gamma-\int_0^{+\infty} \frac{1}{1-e^{-t}} \left ( e^{-st} -e^{-t} \right ) \, dt \ ,
$$
and plugging the integral expression for Euler's constant we finally prove Gauss formula for the logarithmic
derivative of the $\Gamma$-function \cite{WW}, p.246).

\begin{theorem} \textbf{(C.F. Gauss)}
$$
\frac{\Gamma'(s)}{\Gamma (s)}=\int_0^{+\infty} \left (\frac{e^{-t}}{t}-\frac{e^{-st}}{1-e^{-t}}
\right ) \, dt \ .
$$
 \end{theorem}

So, according to our interpretation, Gauss integral formula is equivalent to all Newton relations 
for the roots $e^{-n}$ for all exponents $t\in \RR_+^*$. This interpretations seems new.

\medskip

In a similar way we can obtain Binet formula for the logarithm of the $\Gamma$-function (see the section on the $\Gamma$-function in \cite{WW}), as well as general formulas for higher $\Gamma$-functions, the first example being 
Barnes $\Gamma$-function.  This will be developped in future versions of this article.

\subsection{Explicit formulas for Riemann zeros.}

In this section we apply our Poisson-Newton formula to the Riemann zeta function. We obtain 
a non-classical form of the Explicit Formula in analytic number theory. The classical forms can 
be derived from our distributional formula.

Explicit formulas in analytic number theory go back to the original memoir of Riemann \cite{R} on the 
analytic properties of Riemann zeta function where it is the central point of the derivation of Riemann's asymptotic 
formula for the growth of the number of primes. It relates prime numbers with non-trivial zeros of Riemann zeta function. 
Despite the mystery about the precise location of the non-trivial zeros, many of such formulas were developped at the 
end of the XIXth century and the beginning of the XXth century (see \cite{L}). Later, general explicit formulas
were developed by A.P. Guinand \cite{G}, J. Delsarte \cite{D}, A. Weil \cite{W} and K. Barner \cite{Ba}, these last ones in 
general distributional form. A classical form of this Explicit Formula is the following by K. Barner \cite{Ba} :

\begin{theorem} For an appropriate test function $\varphi$ with Fourier transform $\hat \varphi$ analytic in a large enough strip, 
we have
 
\begin{align*}
\sum_\gamma \hat \varphi (\gamma) =  &\hat \varphi (i/2) + \hat \varphi (-i/2)+ \frac{1}{2\pi} \int_\RR \Psi(t) \hat \varphi(t) \ dt\\
&-\sum_{p,k\geq 1} (\log p) p^{-k/2} \left (\varphi (k\log p) +
\varphi (-k\log p) \right )\, ,
\end{align*}
where the $\gamma's$ run over the non-trivial zeros of $\zeta$, the $p$'sover prim numbers, and 
$$
\Psi(t) =-\log \pi +\Re \left ( \frac{\Gamma'}{\Gamma} (1/4+it/2)\right ) \ .
$$

\end{theorem}

Let's see how one can recover this classical formula from our Poisson-Newton formula.

\medskip

We consider the Riemann zeta function defined for $\Re s >1$ by
$$
\zeta(s)=\sum_{n\geq 1} n^{-s}=\sum_{n\geq 1} e^{-s\log n } \, ,
$$
which is a Dirichlet series with $\lambda_n =\log (n+1)$ and $\sigma_1=1$ in our notation. It
has a meromorphic extension to the complex plane $s\in \CC$ with a single simple pole at $s=1$. 
It has order $o=1$ and convergence exponent $d=2$.

For $\Re s >1$ we have the Euler product which gives the relation of the zeta function with prime numbers,
$$
\zeta(s)=\prod_p (1-p^{-s})^{-1} \, ,
$$
where the product is running over the prime numbers $p$. Thus
$$
-\log \zeta (s) = -\sum_{p, \, k\geq 1} \frac{p^{-ks}}{k} =-\sum_{p, \, k\geq 1} \frac1k
e^{-k(\log p)s} \, .
$$
The vector of fundamental frequencies is $\boldsymbol{\lambda}=(\log 2, \log 3 ,\log 5 , \ldots )$. We have
$b_\bk=0$ when $\bk$ has more than one non-zero entry, and 
$b_\bk=-1/k$ for $\la \boldsymbol{\lambda}, \bk\ra= k\log p$.

The Riemann zeta function has a functional equation with $\sigma^*=1/2$, $\sigma_-=0$ and $\sigma_+=1$. 
We have, using the notations of section 
\ref{sec:functional}, 
\begin{align*}
g(s) &=g(1-s)\\
g(s)&= \chi(s) \zeta(s) \\
\chi(s) &=\pi^{-s/2}\Gamma (s/2) s(s-1), \\ 
\chi_0(s) &=\pi^{-s/2}\Gamma(s/2), \\  
R(s) &=s(s-1). 
\end{align*}

The Riemann zeta function has a single simple pole at $\rho=1$, and simple real zeros at $\rho=-2n$,
for $n=1, 2,\ldots $,
and non-real zeros in the critical strip $\sigma_-=0<\Re s < 1=\sigma_+$, $\rho =1/2+i\gamma$. 
The Riemann Hypothesis
conjectures that $\gamma \in \RR$, i.e., that all non-real zeros have real part $1/2$. These non-real zeros are
conjectured to be simple. Following the tradition we will repeat them according to their multiplicity,
so we may skip the multiplicities $n_\rho =1$ in our subsequent formulas.

The Riemann zeta function is real analytic, and we can apply the symmetric Poisson-Newton formula 
(theorem \ref{thm:symmetric}) and the Poisson-Newton formula with parameters 
(theorem \ref{thm:symmetric_parameters}). More precisely,
in order to get the classical formulas and exploit the functional equation, we apply the Poisson-Newton 
formula with parameters, Corollary \ref{cor:symmetric_parameters}, for $\beta=\sigma^*=1/2$, so $\sigma_1'=
1-1/2=1/2$ and we get

\begin{theorem}

$$
\sum_{\rho} n_\rho e^{(\rho -1/2)|t|} = 2 c_0(\zeta, 1/2)\,  \delta_0 - \sum_{p,k\geq 1} (\log p) p^{-k/2} \left ( \delta_{k\log p} + \delta_{-k\log p} \right ) \ .
$$

\end{theorem}

The contribution at $0$, $c_0(\zeta , 1/2)$ is computed in the Appendix, and we have
\begin{theorem}
$$
c_0(\zeta , 1/2)= -\frac{\log \pi}{2} -\frac{\gamma}{2} -2\log 2 \ .
$$
\end{theorem}

We can compute explicitly the contribution of the real divisor to the distribution on the left handside:
$$
W_0(t)=-e^{|t|/2}+e^{-|t|/2}\sum_{n\geq 1} e^{-2n|t|}=-e^{|t|/2} + e^{-5|t|/2}\frac{1}{1-e^{-2|t|}}
=-e^{|t|/2} + e^{-\frac{3}{2}|t|} \frac{1}{2\sinh |t|} \, .
$$
Note that $-e^{|t|/2} W_0=W(\chi )(t)+W(\chi)(-t)=W(\chi_0)(t)+W(\chi_0)(-t)+W(R)(t)+W(R)(-t)$.

So the associated Poisson-Newton formula on $\RR$ is
\begin{align*}
\sum_\gamma e^{ i\gamma |t|} + W_0(t) &= 2 c_0(\zeta,1/2)\,  \delta_0 - 
\sum_{p,k\geq 1} (\log p) p^{-k/2} \left ( \delta_{k\log p} + \delta_{-k\log p} \right ) \\
&= 2 c_0(\zeta,1/2)\,  \delta_0 -\sum_{n\geq 1} \frac{\Lambda (n)}{\sqrt {n}} \ 
\left ( \delta_{\log n} +\delta_{-\log n} \right ) \, ,
\end{align*}
where $\Lambda$ is the Von Mangoldt function, $\Lambda(p^k)=\log p$ and $\Lambda(n)=0$ if $n$ is not the power of a prime number.

\medskip

For a test function $\varphi \in \cS$ in the Schwartz class, consider its Fourier transform
$$
\hat \varphi (x) =\int_\RR \varphi (t) e^{-ixt} \, dt \, .
$$
Observe that
$$
\hat \varphi (\gamma) =\int_\RR \varphi (t) e^{-i \gamma t} \ dt =\int_{\RR_+} \left (\varphi (t) e^{-i \gamma t} +
\varphi (-t) e^{-i (-\gamma ) t} \right )\, dt  \, .
$$
By the real analyticity of $\zeta(s)$, the set of non-trivial zeros is real symmetric,
$(\gamma )=(-\gamma )$, hence
$$
\sum_\gamma \hat \varphi (\gamma) =\int_{\RR_+} (\varphi (t)  +
\varphi (-t))  \left (\sum_\gamma e^{i \gamma  t}\right ) \, dt  \, .
$$

Thus applying now our Poisson-Newton formula to the test function $\varphi$ we get
$$
\sum_\gamma \hat \varphi (\gamma) + W_0[\varphi]= 2 c_0(\zeta,1/2) \varphi (0) 
-\sum_{p,k\geq 1} (\log p) p^{-k/2} (\varphi (k\log p) +
\varphi (-k\log p) )\, ,
$$
where $W_0[\varphi]$ is the functional
$$
W[\varphi]= \int_{\RR }  W_0 (t) \varphi (t) \, dt \, .
$$
We compute more precisely this functional. We have
\begin{align*}
W_0 &=-e^{-|t|/2} (W(\chi )(t)+W(\chi)(-t)) \\
&=-e^{-|t|/2} ( W(\chi_0)(t)+W(\chi_0)(-t)+W(R)(t)+W(\bar R)(-t)) \ .
\end{align*}

We assume that $\hat \varphi$ is holomorphic in a neighborhood 
of the strip $|\Im t |\leq 1/2$, then we have by the general symmetric Poisson-Newton formula 
(or by direct computation)
\begin{align*}
\langle -e^{-|t|/2} (W(R)(t)+W(R)(-t)) , \varphi \rangle &=-\int_\RR e^{-|t|/2} (e^{|t|} +1)\varphi(t) \, dt\\
&=-\int_\RR 2 \cosh (|t|/2) \varphi (t) \, dt \\
&=-\int_\RR 2 \cosh (t/2) \varphi (t) \, dt \\
&=-\int_\RR \varphi (t) e^{t/2}\, dt - \int_\RR \varphi (t) e^{-t/2}\, dt \\
&= -\hat \varphi (i/2) - \hat \varphi (-i/2)  \ .
\end{align*}
Now, again using the general symmetric Poisson-Newton formula, more precisely, 
Corollary \ref{cor:symmetric-parameters-general} with $\alpha=1$ and $\beta=1/2$ 
applied to $\chi_0$ that is real analytic,   we have
$$
-e^{-|t|/2} (W(\chi_0)(t)+W( \chi_0)(-t)) = 2 c_0(\chi_0 , 1/2) \delta_0 
+\cL^{-1}_{1/2}\left (2 \Re \left (\frac{\chi'_0}{\chi_0}\right ) \right ) \ .
$$
And using Proposition \ref{functional_eq_0structure},
\begin{equation}\label{eqn:c2}
c_0(\chi_0 , 1/2)+c_0(\zeta , 1/2) =0 \ ,
\end{equation}
thus the Poisson-Newton formula applied to an appropriate test function $\varphi$ is
\begin{align*}
\sum_\gamma \hat \varphi (\gamma) =  &\hat \varphi (i/2) + \hat \varphi (-i/2) +\left \langle 
\cL^{-1}_{1/2}\left (2 \Re \left (\frac{\chi'_0}{\chi_0}\right ) \right ), \varphi \right \rangle \\
&-\sum_{p,k\geq 1} (\log p) p^{-k/2} (\varphi (k\log p) +
\varphi (-k\log p) )\, ,
\end{align*}

Now, we have
$$
\frac{\chi'_0 (s)}{\chi_0 (s)} =-\frac12 \log \pi +\frac12 \frac{\Gamma'(s/2)}{\Gamma(s/2)} \ ,
$$
so
$$
\left \langle 
\cL^{-1}_{1/2}\left (2 \Re \left (\frac{\chi'_0}{\chi_0}\right ) \right ), \varphi \right \rangle =
\frac{1}{2\pi} \int_\RR \Psi(t) \hat \varphi(t) \ dt \ ,
$$
where 
$$
\Psi(t) =-\log \pi +\Re \left ( \frac{\Gamma'}{\Gamma} (1/4+it/2)\right ) \ .
$$

Thus we recover the classical form of the Explicit formula stated at the beginning of the section. Historically this form is due 
to Barner that gave a new form of the Weil functional. Barner's derivation is based on an integral formula, Barner formula, that 
can be directly derived from our general Poisson-Newton formula.

Note that our ``explicit formula'' appears more concise that the classical formulation, and even more if we use Corollary 
\ref{cor:symmetric_parameters} with $\beta=0$

\begin{theorem}
We have
$$
\sum_{\rho} n_\rho e^{\rho |t|} = 2 c_0(\zeta, 0)\,  \delta_0 - \sum_{p,k\geq 1} (\log p)  \left ( \delta_{k\log p} + \delta_{-k\log p} \right ) \ ,
$$
and 
$$
c_0(\zeta , 0) = -\log (2\pi ) \ .
$$
\end{theorem}

We can compute $c_0(\zeta , 0)$ from the known Hadamard factorization of Riemann zeta function. We have (see \cite{T} p.31):
$$
\zeta(s) = \frac{e^{bs}}{2(s-1)\Gamma(s/2+1)} \prod_\rho  E_1(s/\rho) = \frac{e^{bs}}{s(s-1)\Gamma(s/2)} \prod_\rho  E_1(s/\rho)\ ,
$$
where the product is over the non-trivial zeros and 
$$
b=\log (2\pi )-1-\gamma /2 \ .
$$ 

Now, we have
$$
\frac{1}{s-1} =- e^{s} \left (E_1(s/1)\right )^{-1}\ ,
$$
thus
$$
c_0(\zeta , 0)=-Q_\zeta =-b  -1+c_0\left (1/\Gamma(s/2) , 0\right ) \ .
$$
But from the Hadamard factorization of the $\Gamma$-function we have 
$$
\frac{1}{\Gamma(s/2)} = \frac{s}{2} e^{\frac{\gamma}{2} s} \prod_{n\geq1} E_1(s/(-2n)) \ ,
$$
thus 
\begin{equation}\label{eqn:c3}
c_0 \left (1/\Gamma(s/2) , 0\right ) = -\frac{\gamma}{2} \ ,
\end{equation}
and
\begin{equation}\label{eqn:c1}
c_0(\zeta, 0)=-\log ( 2 \pi ) \ .
\end{equation}

We have the final formula:
\begin{theorem}
$$
\sum_{\rho} n_\rho e^{\rho |t|} = -2 \log (2\pi )\,  \delta_0 - \sum_{p,k\geq 1} (\log p)  
\left ( \delta_{k\log p} + \delta_{-k\log p} \right ) \ .
$$
\end{theorem}

\begin{remark}{\textbf{(Newton relations interpretation)}}
Again, by our general interpretation, the Explicit formula appears as the ``Newton relations'' which links the non-trivial zeros with 
the primes, the primes playing a similar role than coefficients in Newton formulas.
\end{remark}

\begin{remark}
{\textbf{(General Explicit Formulas)}}
The derivation given  
of the classical distributional Explicit Formula is general 
and applies to any Dirichlet series of order $1$ with the required conditions. In this sense the Poisson-Newton formula can be seen as 
the general Explicit Formula associated to a Dirichlet series. The structure at $0$ needs to be computed in general. 
But when we have a functional equation, one can apply the Poisson-Newton formula with the parameter 
well chosen so that the structure at $0$ vanishes from the formula (as we have done in the previous section for the 
Riemann zeta function). The divisor on the left cone gives the general ``Weil functional'' and again, by application of the general 
Poisson-Newton formula with parameters and using Hadamard regularization for this $\theta$-distribution we get a general Barner integral formula 
for the functional. Thus we get a general Explicit Formula with the same structure as for the classical one for Riemann zeta function.
\end{remark}

\subsection{General Guinand equation.}

The Newton-Cramer distribution $W(f)$ can be naturally be decomposed in the form of an hiperfunction $W(f)=W_+(f)+W_-(f)$ by separating 
zeros with positive and negative imaginary parts. Both $W_+(f)$ and $W_-(f)$ are analytic functions on cones sharing $\RR_+$ in its boundary.
If $f$ is real analytic, then its zeros are symmetric with respect to the real axes, giving a relation between  $W_+(f)$ and $W_-(f)$. This 
relation plugged into
$$
W(f)=W_+(f)+W_-(f) \ ,
$$
givesa functional equation for $W_+(f)$ which generalizes Guinand functional equation for the Cramer function associated to Riemann zeta-function.

Therefore this proves that we have general Guinand equations for the generalization of the Cramer function for general real analytic Dirichlet series.

\subsection{Selberg Trace formula.}

It is well known that Selberg trace formula was developed 
by analogy with the Explicit Formulas in analytic number theory 
and that this was the original motivation by Selberg (see \cite{S}, \cite{CV}). 
In this section we explain this folklore analogy by showing that 
Selberg Trace Formula results from the Poisson-Newton formula applied to Selberg zeta function. 
The approach is very similar 
to that of the previous section and we have a unified treatment of both formulas. 
The only relevant difference is that Selberg zeta function 
is of order $2$.

\medskip

We consider a compact Riemannian surface $X$ of genus $h\geq 2$ with a metric of constant negative
curvature. Let $\cP$ be
the set of primitive geodesics. The Selberg zeta function is defined in the half plane $\Re s>1$ by
the Euler product
$$
\zeta_X(s)=\prod_{p\in \cP}\prod_{k\geq 0} \left ( 1-e^{\tau(p)(s+k)}\right ) \ ,
$$
where $\tau(p)$ is the length of the geodesic $p$.

We have 
\begin{align*}
-\log \zeta_X(s) &= \sum_p \sum_{k\geq 0} \sum_{l\geq 1} \frac1l \ e^{-\tau(p)(s+k) l}\\
&=\sum_{p , l\geq 1} \frac1l\  e^{-\tau(p) ls} \ \frac{1}{1-e^{\tau(p) l}}\\
&=\sum_{p , l\geq 1} \frac1l \ e^{-\tau(p) l/2} \ \frac{1}{2 \sinh (\tau(p) l/2)}\  e^{-\tau(p) l s}
\end{align*}

Thus we compute the coefficients
$$
b_{p,l} =\frac1l \ e^{-\tau(p) l/2} \ \frac{1}{2 \sinh (\tau(p) l/2)} \ ,
$$
and the frequencies
$$
\langle {\boldsymbol {\lambda }} , (p,l)\rangle = \lambda_{p,l} = \tau(p) l \ .
$$

One of the fundamental results of the theory is that $\zeta_X$ has a meromorphic extension 
to the complex plane of order $2$, exponent of convergence $d=3$, thus genus $g=2$ 
by Corollary \ref{cor:genus}, has a functional equation with $\sigma^*=1/2$, and its zeros are the following (see \cite{V}, p.129):

\begin{itemize}
 \item Trivial zeros at $s=-k$ with $k=0,1,2,\ldots$ with multiplicity $2(h-1)(2k+1)$.
 \item Non-trivial zeros $s=1/2\pm i \gamma_n$, $n=0,1,2,\ldots$, where $1/4+\gamma_n^2$ are the
eigenvalues of the positive Laplacian $-\Delta_X$ on $X$ counted with multiplicity. The lowest eigenvalue
$0$ yields two zeros, $s=1$ that is simple, and the trivial zero $s=0$ with multiplicity $2(h-1)$ (we 
exclude the case of  $1/4$ as eigenvalue). 

\end{itemize}

For $n<0$ write $\gamma_n=-\gamma_{-n}$.
Therefore the Newton-Cramer distribution decomposes as
$$
W(\zeta_X)=V(\zeta_X) + W_0 (\zeta_X)\ ,
$$
where $W_0$ is the contribution of the trivial zeros and $V$ the contribution of the non-trivial ones. 
We compute on $\RR^*$ with $\beta =1/2$
\begin{align*}
\widehat W_0(\zeta_X, 1/2) (t) &=\sum_{n\in \ZZ}  2(h-1) (2n+1) e^{(-n-1/2)|t|} \\
&=4(h-1) \sum_{n\geq 0} (n+1/2) e^{-(n+1/2) |t|} \\
&=-4(h-1)\frac{d}{d|t|} \left (\frac{1}{2\sinh (|t|/2)} \right )\\
&=(h-1) \frac{\cosh(t/2)}{(\sinh(t/2))^2} \ .
\end{align*}
And we have
$$
\widehat V(\zeta_X, 1/2) (t) =  \sum_{n\in \ZZ} e^{i\gamma_n |t|}=2 \sum_{n\geq 0}\cos(\gamma_n t) \ . 
$$

Now we apply the symmetric Poisson-Newton formula with parameter (Corollary \ref{cor:symmetric_parameters})
with $\beta=1/2$, and we get
\begin{align*}
 \widehat W_0(\zeta_X, 1/2) + \widehat V(\zeta_X, 1/2) &= 2 c_0(\zeta_X , 1/2) + \sum_{p, l\in \ZZ^*} 
| \langle {\boldsymbol {\lambda }} , (p,l)\rangle |  
e^{-| \langle {\boldsymbol {\lambda }} , (p,l)\rangle |/2} \,  b_{p,|l|} \,
\delta_{\langle {\boldsymbol {\lambda }} , (p,l)\rangle} \\
2 \sum_{\gamma} e^{i\gamma t} + (h-1) \frac{\cosh(t/2)}{(\sinh(t/2))^2} 
&= 2 \sum_{p, l\in \ZZ^*} 
\frac{\tau(p)}{4 \sinh(\tau(p) |l|/2)} \,
\delta_{\tau(p) l} \ ,
\end{align*}
where we used $c_0(\zeta_X , 1/2) =0$ by Proposition \ref{functional_eq_0structure}.

This yields the classical Selberg Trace Formula as stated in \cite{CV}:

\begin{theorem} {\textbf {(Selberg Trace Formula)}} We have
$$ 
\sum_{\gamma} e^{i\gamma t} =-\frac12 (g-1) \frac{\cosh(t/2)}{(\sinh(t/2))^2} 
+ \sum_{p, l\in \ZZ^*} 
\frac{\tau(p)}{4 \sinh(\tau(p) |l|/2)} \,
\delta_{\tau(p) l} \ .
$$
\end{theorem}

We can now manipulate the integral expression for the ``Weil functional'' \`a la Barner, using the 
general Poisson-Newton formula as we have done in the previous section, etc. These computations will be done elsewhere. 

\begin{remark}
{\textbf {(Gutzwiller Trace formula)}}
The Selberg trace formula is just a particular case of the Gutzwiller Trace formula in Quantum Chaos (see \cite{GU}). We see 
that in general Gutzwiller Trace Formula, that is the central formula in quantum chaos, 
results from the application of the Poisson-Newton formula to the dynanmical zeta function of the Dynamical System
when this zeta function has an analytic extension to the whole complex plane. Thus non-trivial zeros are related to the quantum energy levels
and the frequencies to the classical periodic orbits. 
\end{remark}

\subsection{Lifting formulas.}

The ``lifting formulas'' developped in this section are examples of Poisson-Newton formulas 
for Dirichlet series of infinite order. They have a transalgebraic meaning that will be developped 
elsewhere.

We have normalized our Dirichlet series by $a_0=1$, but we can carry out the same analysis in general for
$$
f(s)= a_0+\sum_{n\geq 1} a_n \, e^{-\lambda_n s} \ ,
$$
with $a_0\not= 0$.

We can write
$$
f(s)= a_0 \left ( 1+\sum_{n\geq 1} \frac{a_n}{a_0} \, e^{-\lambda_n s} \right ) \ ,
$$
and we have the associated Poisson-Newton formula in $\RR_+^*$
$$
\sum_\rho n_\rho e^{\rho t} = \sum_\bk \langle \boldsymbol\lambda , \bk \rangle \frac{b_\bk}{a_0^{||\bk||}}
\, \delta_{\langle \boldsymbol\lambda , \bk \rangle}  \ ,
$$
where the first sum is over the zeros $(\rho)$ of $f$.
But we can also write
$$
f(s) = (a_0-1)+ \left ( 1 +\sum_{n\geq 1} a_n \, e^{-\lambda_n s} \right )=(a_0-1)+g(s) \ .
$$
Note that the zeros $\{\eta\}$ of $g$ are the preimages by $f$ of $a_0-1$. Hence we have proved

\begin{proposition}
 We have in $\RR_+^*$, where the first sum is taken with multiplicity
$$
\sum_{\eta ; f(\eta )=a_0-1}e^{\eta t} = \sum_\bk \langle \boldsymbol\lambda , \bk \rangle \frac{b_\bk}{a_0^{||\bk||}}
\, \delta_{\langle \boldsymbol\lambda , \bk \rangle}  \ .
$$
\end{proposition}

Observe that when $||\bk||=1$, say $\bk=(0,\ldots , 0, 1, 0, \ldots )$ with $1$ at the $j$-th place, then
$$
b_\bk =-a_j \ .
$$
Now adding these Poisson-Newton formulas for $a_0=1, 2, 3, \ldots$ we get

\begin{corollary}
In $\RR_+^*$ we have
 $$
\sum_{m=0}^{+\infty }\left (\sum_{\rho ; f(\rho )=m} e^{\rho t} +\sum_{n=1}^{+\infty} \lambda_n {a_n}
 \delta_{\lambda_n} \right )= \sum_{\bk \in \Lambda ; ||\bk||\geq 2} 
 \langle \boldsymbol\lambda , \bk \rangle b_\bk \, \zeta (||\bk||) \, 
 \delta_{\langle \boldsymbol\lambda , \bk \rangle}  \ .
$$
\end{corollary}

Or also

\begin{corollary} In $\RR_+^*-\{\lambda_n\}$, we have
 $$
\sum_{\rho \in  f^{-1}(\ZZ)} e^{\rho t} = \sum_{\bk \in \Lambda ; ||\bk|| \text{ even}}
\langle \boldsymbol\lambda , \bk \rangle b_\bk \, (2 \, \zeta (||\bk||) -1)
\, \delta_{\langle \boldsymbol\lambda , \bk \rangle}  \ .
$$
\end{corollary}

\section{Appendix}

In this appendix we determine the relation between $Q_{f,\sigma}$ and $Q_{f,0}$. In particular,  
this the variation of the coefficient $c_0(f,\sigma)$ from $c_0(f,0)$ and we apply this to 
compute $c_0(\zeta,1/2)$.

Let $f$ be of finite order and consider the Hadamard factorization of $f$ (see \cite{A} p.208)
$$
f(s)=s^{n_0} e^{Q_f(s)} \prod_{\rho \not=0 } E_m (s/\rho )^{n_\rho} \ ,
$$
where $m=d-1\geq 0$ is minimal for the convergence of the product with
$$
E_m(z)=(1-z) e^{z+\frac12 z^2 +\ldots +\frac1m z^m} \ ,
$$
and $Q_f$ is a polynomial uniquely defined up to the addition of an integer multiple of $2\pi i$. 
Consider now $\sigma\in\CC$ and the corresponding Hadamard factorization centered at $\sigma$,
$$
f(s)=(s-\sigma)^{n_\sigma} e^{Q_{f,\sigma}(s)} \prod_{\rho \not=\sigma } E_m \left(\frac{s-\sigma}{\rho-\sigma} 
\right)^{n_\rho} \ .
$$

We want to understand the difference between these two factorizations. We take logarithmic derivatives to get
 \begin{align*}
  & \frac{n_\sigma}{s-\sigma} + Q'_{f,\sigma} +  \sum_{\rho\neq 0,\sigma} n_\rho 
\frac{(s-\sigma)^m}{(\rho-\sigma)^m} \frac{1}{s-\rho}
 + n_0 \frac{(s-\sigma)^m}{(-\sigma)^m} \frac{1}{s} = \\
& = \frac{n_0}{s} + Q'_{f} +  \sum_{\rho\neq 0,\sigma} n_\rho 
\frac{s^m}{\rho^m} \frac{1}{s-\rho}
 + n_\sigma \frac{s^m}{\sigma^m} \frac{1}{s-\sigma}
 \end{align*}
Therefore
 \begin{align*}
  Q'_{f,\sigma}  -   Q'_{f} = & n_0\frac{(-\sigma)^m -(s-\sigma)^m}{(-\sigma)^m s} + n_\sigma
\frac{s^m -\sigma^m}{\sigma^m (s-\sigma)}  \\
&+ \sum_{\rho\neq 0,\sigma} n_\rho 
\frac{s^m(\rho-\sigma)^m- (s-\sigma)^m \rho^m}{\rho^m(\rho-\sigma)^m} \frac{1}{s-\rho}
 \end{align*}
For $m=1$ this reduces to 
 \begin{align} \label{eqn:c4}
  Q'_{f,\sigma}  -   Q'_{f} = & \frac{n_0}{\sigma} + \frac{n_\sigma}{\sigma} + \sum_{\rho\neq 0,\sigma} n_\rho 
\frac{-\sigma}{\rho(\rho-\sigma)}
 \end{align}
For $m=2$, it becomes
 \begin{align*}
  Q'_{f,\sigma}  -   Q'_{f} = & n_0\frac{2\sigma-s}{\sigma^2} + 
n_\sigma\frac{s+\sigma}{\sigma^2} + \sum_{\rho\neq 0,\sigma} n_\rho 
\frac{(-2\rho\sigma+\sigma^2) s +\rho\sigma^2}{\rho^2(\rho-\sigma)^2}\, .
 \end{align*}

We also can calculate the discrepancy when we change from $\sigma$ to $\sigma'$
by considering
 $$
Q_{f.\sigma'} - Q_{f.\sigma} =(Q_{f.\sigma'} - Q_{f}) - (Q_{f.\sigma} - Q_{f} ).
$$
For $m=1$, it is of the form $A\, s + B$, where
\begin{align*}
 A &= \frac{n_{\sigma'}}{(\sigma'-\sigma)^2} -\frac{n_{\sigma}}{(\sigma'-\sigma)^2}
  +\sum_{\rho\neq \sigma,\sigma'} n_\rho \frac{(\sigma+\sigma'-2\rho)(\sigma'-\sigma)}{(\rho-\sigma)^2(\rho-\sigma')^2} \\
 B &= n_{\sigma'} \frac{\sigma'-2\sigma}{(\sigma'-\sigma)^2} - n_{\sigma} \frac{\sigma-2\sigma'}{(\sigma'-\sigma)^2}
  +\sum_{\rho\neq \sigma,\sigma'} n_\rho \frac{(\rho\sigma+\rho\sigma'-2\sigma\sigma')
(\sigma'-\sigma)}{(\rho-\sigma)^2(\rho-\sigma')^2}\, .
\end{align*}

\medskip

 To compute
$c_0(\zeta,1/2)$, note that $c_0(\zeta,1/2)=-c_0(\chi_0,1/2)$ by (\ref{eqn:c2}). The value of
 $$
 c_0(\chi_0,0)=\frac{\log \pi}{2}+\frac{\gamma}{2}
 $$
follows from (\ref{eqn:c3}). The zeros of $\chi_0$ are the negative integers $-n$, $n\geq 0$, and 
are simple. Hence 
the formula (\ref{eqn:c4}) reads (for $\sigma$ not a pole of $\chi_0$)
 \begin{align*}
 -c_0(\chi_0,\sigma)+c_0(\chi_0,0) &=-\frac{1}{\sigma} + \sum_{n=1}^\infty (-1) 
\frac{-\sigma}{(-n)(-n-\sigma)} \\
 &= -\frac{1}{\sigma}+ \sum_{n=1}^\infty \left(\frac1n -\frac1{n+\sigma}\right) \\
  &=  \frac{\Gamma'(\sigma)}{\Gamma(\sigma)} +\gamma ,
  \end{align*}
where the last formula follows from the expression for the logarithmic derivative of the the $\Gamma$-function, 
the digamma function $\psi$,
$$
\psi(s)=\frac{\Gamma' (s)}{\Gamma (s)} =-\frac1s -\gamma +\sum_{n=1}^{+\infty} \left (\frac{1}{n} -\frac{1}{n+s} \right ) \ ,
$$
which results from its Hadamard factorization.
 
Finally, We get

\begin{theorem} We have, for $\sigma \notin -\ZZ$
$$
c_0(\chi_0, \sigma)=\frac{\log \pi}{2} -\frac{\gamma}{2} - \psi (\sigma) \ .
$$

\end{theorem}

In particular, for $\sigma=1/2$, we have (see \cite{AS} entry 6.3.3 p. 258)
$$
\psi(1/2)=-2\log 2-\gamma \ .
$$

\begin{theorem} We have,

$$
c_0(\zeta , 1/2)=-c_0(\chi_0, 1/2)= -\frac{\log \pi}{2} -\frac{\gamma}{2} -2\log 2 \ .
$$

\end{theorem}

\end{document}